\documentclass[12pt, a4paper]{amsart}
\usepackage[margin = 2.9cm]{geometry}
\usepackage[title]{appendix}
\usepackage{mathtools}
\usepackage{amsmath}
\usepackage{systeme}
\usepackage{paralist}
\usepackage{graphics} 
\usepackage{epsfig} 
\usepackage{graphicx}  
\usepackage{epstopdf}
 \usepackage[colorlinks=true]{hyperref}
\hypersetup{urlcolor=blue, citecolor=red}

\newtheorem{thm}{Theorem}[section]

\newtheorem{lm}{Lemma}
\newtheorem{pp}{Proposition}
\newtheorem*{claim}{Claim}

\theoremstyle{definition}
\newtheorem{df}{Definition}
\newtheorem{remark}{Remark}

\newtheorem{example}{Example}
\newcommand*\diff{\mathop{}\!\mathrm{d}}

\newcommand\numberthis{\addtocounter{equation}{1}\tag{\theequation}}

\DeclareMathOperator{\WF}{WF}

\DeclareMathOperator{\sgn}{sgn}

\DeclareMathOperator{\Tr}{Tr}

\begin{document}
\title[Rayleigh and Stoneley Waves in Linear Elasticity]{Rayleigh and Stoneley Waves in Linear Elasticity}
\author[Y. Zhang]{Yang Zhang}
\address{Purdue University \\ Department of Mathematics}
\email{zhan1891@purdue.edu}
\thanks{Partly supported by NSF Grant DMS-1600327}
\maketitle
\begin{abstract}
We construct the microlocal solutions of Rayleigh and Stoneley waves in linear elasticity with variable coefficients and a curved boundary (interface). We compute the direction of the polarization and show the retrograde elliptical motion of both of them.  
\end{abstract}

\section{Introduction}
Rayleigh waves in linear elasticity are a type of surface waves. 
They are first studied by Lord Rayleigh in \cite{Rayleigh1885} and can be the most destructive waves in an earthquake.
They propagate along a traction-free boundary and decay rapidly into the media.
By the geophysical literatures, Rayleigh waves have a retrograde elliptical particle motion for shallow depth in the case of flat boundary and homogeneous media, see \cite{1904a,KeiitiAki2002}.
Stoneley waves are a type of interface waves that propagate along the interface between two different solids. 
They are first predicted in \cite{1924a}. 
Roughly speaking, Rayleigh waves can be regarded as a special (limit) case of Stoneley waves.
Both geophysical and mathematical works have been done for these two kinds of waves, see \cite{Rayleigh1885,Scholte1947,Chadwick1994,Stefanov2000, Babich2004, Hansen2012,Kirpichnikova2008, 1904a,Taylor1979,Yamamoto1989,Hansen2011,Hansen2014,Chadwick1994, 1985a, Ting2014, Hoop2020, Stefanov1996, Stefanov1995} and their references. 
Most geophysical works on them are considering specific situations, for example, the case of flat boundaries, plane waves, or homogeneous media.
The propagation phenomenon
of Rayleigh waves in an isotropic elastic system is first studied by Michael Taylor in \cite{Taylor1979} from a microlocal analysis point of view. 
Kazuhiro Yamamoto in \cite{Yamamoto1989}  shows the existence of Stoneley waves as the propagation of singularities in two isotopic media with smooth arbitrary interfaces.
S\"{o}nke Hansen in \cite{Hansen2011} derives the Rayleigh quasimodes by the spectral factorization methods for inhomogeneous anisotropic media with curved boundary 
and then in \cite{Hansen2014} shows the existence of Rayleigh waves by giving ray series asymptotic expansions in the same setting. 
In particular, the author derives the transport equation satisfied by the leading amplitude which represents the term of highest frequency. 
In \cite{Hoop2017}, the authors develop a semiclassical analysis of elastic surface wave generated by interior (point) source and 
in \cite{Hoop2020} the inverse spectral problem of Rayleigh waves is studied. 
Most recently in \cite{Stefanov2019}, the authors describe the microlocal behavior of solutions to the transmission problems in isotropic elasticity with curved interfaces. 
Surface waves are briefly mentioned there as possible solutions of evanescent type which propagate on the boundary, see \cite[\S 8.2]{Stefanov2019}. 

In this work, we describe the microlocal behaviors of Rayleigh waves and Stoneley waves for an isotropic elastic system with variable coefficients and a curved boundary (interface).
We construct the microlocal solutions of these two waves 
and compute the direction of their polarization explicitly. 
The magnitude of the polarization depends on an amplitude constructed by the geometric optics, which satisfies a transport equation. 
We explain how to compute the zero order term involved in the transport equation in the appendix.
In particular, we show the retrograde elliptical particle motion in both waves as an analog to the flat case.
Essentially, the existence of Rayleigh waves comes from the nonempty kernel of the principal symbol of the Dirichlet-to-Neumann map (DN map) $\Lambda$ in the elliptic region. 
In Section \ref{section_pre}, we briefly state some relevant results in \cite{Stefanov2019} that we are going to use later.
In Section \ref{Rayleigh}, based on the analysis in  \cite{Stefanov2019}, 
one can see the Rayleigh waves are corresponding to the solution to $\Lambda u = l$ on the boundary, where $l$ is a source microlocally supported in the elliptic region.
Next, inspired by the diagonalization of the Neumann operator for the case of constant coefficients in \cite{Stefanov2000}, we diagonalize $\Lambda$ microlocally up to smoothing operators by a symbol construction in \cite{Taylor2017}. 
The DN map $\Lambda$ is a matrix-valued pseudodifferential operator with the principal symbol $\sigma_p(\Lambda)$ in (\ref{symbol_dnmap}) and the way to diagonalize it is not obvious.
One can see the diagonalization is global and
it gives us a system of one hyperbolic equation and two elliptic equations on the boundary with some metric. 
The solution to this system applied by a $\Psi$DO of order zero serves as the Dirichlet boundary condition on the timelike boundary $\mathbb{R} \times \partial \Omega$ of the elastic system, and then the Rayleigh wave can be constructed in the same way as the construction of the parametrix for elliptic evolution equations, as the Cauchy data is microlocally supported in the elliptic region. 
The wave front set and the direction of the microlocal polarization of the Rayleigh waves on the boundary can be derived during the procedure and 
they explain the propagation of Rayleigh waves and show a retrograde elliptical particle motion.
These results are based on the diagonalization of the DN map.
In Section \ref{section_R_cauchy}, we derive the microlocal Rayleigh waves on the boundary if we have the Cauchy data at $t=0$. 
The polarization is given in Theorem \ref{thm_cauchy} and the leading term shows a retrograde elliptical motion of the particles, same as that of the case of homogeneous media in \cite{1904a,KeiitiAki2002}, as is explained after Theorem \ref{thm_cauchy}. 
In Section \ref{section_R_ih}, the inhomogeneous problem, i.e., when there is a source on the boundary, is studied and the microlocal solution and polarization are presented in Theorem \ref{thm_ihm}.
In the second part of this work, Stoneley waves are analyzed in a similar way with a more complicated system on the boundary.
We construct the microlocal solutions without justifying the parametrix. 
It still remains to prove that the exact solution has the same microlocal behavior as the solution we construct. 
The main results of Stoneley waves for Cauchy problems is in Theorem \ref{thm_cauchy_S} and for inhomogeneous problems is in Thereon \ref{thm_ihm_S}.
The microlocal Stoneley waves derived there have similar patterns as that of Rayleigh waves and one can show the leading term shows a retrograde elliptical motion of the particles as well.

\subsection*{Acknowledgments}
The author would like to thank Prof. Plamen Stefanov for suggesting this problem and for lots of helpful discussions with him throughout this project, and to thank Prof. Vitaly Katsnelson, Prof. Mark Williams for helpful suggestions on the transmission problems part.

\section{Preliminaries}\label{section_pre}
Suppose $\Omega \subset \mathbb{R}^3$ is a bounded domain with smooth boundary.
Suppose the density $\rho$ and the Lam\'e parameters $\mu, \nu$ are smooth functions depending on the space variable $x$, and even the time $t$. 

In this section, we recall some notations and results in \cite{Stefanov2019}. 
For a fixed point $x_0$ on the boundary, one can choose the semigeodesic coordinates $x = (x', x^3)$ such that the boundary $\partial \Omega$ is locally given by $x^3 = 0$.  
For this reason, we view $u$ as a one form and write the elastic system in the following invariant way in presence of a Riemannian metric $g$. 
Let $\nabla$ be the covariant differential in Riemannian geometry. 
We define the symmetric differential $\diff^s$ and the divergence $\delta$ as
\begin{align*}
(\diff^s u)_{ij} = \frac{1}{2}(\nabla_i u_j + \nabla_j u_i), \quad (\delta v)_i = \nabla^j v_{ij}, 
\quad \delta u = \nabla^i u_i,
\end{align*}
where $u$ is a covector field and $v$ is a symmetric covariant tensor field of order two.
The stress tensor is given by 
\[
\sigma(u) = \lambda(\delta u) g + 2 \mu \diff^s u.
\]
Then the operator $E$ and the normal stress are 
\[
Eu=  \rho^{-1} \delta \sigma(u) = \rho^{-1} (\diff(\lambda \delta u) + 2 \delta(\mu \diff^s u)), 
\quad N u = \sum_j \sigma_{ij}(u) \nu^j|_{\partial \Omega},
\]
where $\nu^j$ is the outer unit normal vector on the boundary.
The elastic wave equation can be written as
\begin{align*}
u_{tt} = Eu
\end{align*}
and near some fixed $(x_0, \xi^0)$ one can decouple this system up to smoothing operators by a $\Psi$DO $U$ of order zero such that
\begin{align*}
U^{-1} E U = 
\begin{pmatrix}
c_s^2 \Delta_g + A_s & 0\\
0 & c_p^2 \Delta_g + A_p\\ 
\end{pmatrix}
\mod \Psi^{-\infty},
\end{align*}
where $A_s$ is a $2 \times 2$ matrix $\Psi$DO of order one, $A_p$ is a scalar $\Psi$DO of order one, and the S wave and P wave speed are
\begin{align*}
c_s = \sqrt{\mu/\rho}, \quad c_p = \sqrt{(\lambda + 2 \mu)/\rho}.
\end{align*}
Let $w =(w^s, w^p) = U^{-1}u$. 
Then the elastic system decouples into two wave equations.
This decoupling indicates that the solution $u$ has a natural decomposition into the S wave and P wave modes.
\subsection{The boundary value problem in the elliptic region}\label{section_pre_bvp}
When we solve the boundary value problems for the elastic system, the construction of the microlocal outgoing solution depends on where the wave front set of the Cauchy data is. 
The Rayleigh waves happen when there is a free boundary and the singularities of the Cauchy data on the boundary are in the \textbf{elliptic region} $\tau^2 < c_s^2 |\xi'|_g^2$. 

In this case, given the boundary data $f = u|_{x^3 = 0}$, first we get $w_b \equiv w|_{x^3 =0}$ by considering the restriction operator $U_{out}$ of the $\Psi$DO $U$ to the boundary, which maps $w|_{x^3 =0}$ to $f$. 
It is shown that $U_{out}$ is an elliptic one and therefore it is microlocally invertible in \cite{Stefanov2019}.
Then we seek the outgoing microlocal solution $w$ to the two wave equations with the boundary data.
Since the wave front set of the boundary data is in the elliptic region, the Eikonal equations have no real valued solutions. Instead, the microlocal solution is constructed by a complex valued phase function, see \cite[\S5.3]{Stefanov2019} and Section \ref{u_construct} for more details. After we construct $w$, we have $u = U^{-1}w$ as the microlocal solution to the elastic system.
\section{Rayleigh Wave}\label{Rayleigh}
The main goal of this section is to construct the microlocal solution of Rayleigh waves and to analyze their microlocal polarization. 
We follow Denker's notation to denote the vector-valued distributions on a smooth manifold $X$ with values in $\mathbb{C}^N$ by $\mathcal{D}'(X, \mathbb{C}^N)$. Similarly $\mathcal{E}'(X, \mathbb{C}^N)$ is the set of distributions with compact support in $X$ with values in $\mathbb{C}^N$. 

Suppose only for a limited time $0 <t< T$ there is a source on the timelike boundary 
$\Gamma = \mathbb{R}_t \times \partial \Omega$.
Let $u$ be an outgoing solution to the boundary value problem with the inhomogeneous Neumann boundary condition, i.e.
\begin{equation}\label{elastic_eq}
\begin{cases}
u_{tt} - Eu  = 0 & \mbox{in } \mathbb{R}_t \times \Omega,\\
N u = l &\mbox{on } \Gamma,\\
u|_{t<0} = 0,
\end{cases}
\end{equation}
where $l(t,x') \in \mathcal{E}'((0,T) \times \mathbb{R}^2, \mathbb{C}^3)$ is microlocally supported in the elliptic region.
Let $\Lambda$ be the \textbf{Dirichlet-to-Neumann map} (DN map) which maps the Dirichlet boundary data $u|_\Gamma$ to the Neumann boundary data $Nu|_\Gamma$.
Then solving (\ref{elastic_eq}) microlocally is equivalent to solving the Dirichlet boundary value problem $u|_\Gamma = f$,
with $f$ being the outgoing solution on $\Gamma$ to 
\begin{equation}\label{dn_eq}
\Lambda f = l, \quad \text{} f|_{t<0} = 0.
\end{equation}
More specifically, as long as we solve $f$ from (\ref{dn_eq}), we can construct the 
microlocal outgoing solution $u$ to (\ref{elastic_eq}) as an evanescent mode with a complex valued phase function.
The construction of the microlocal solution to the boundary value problem with wave front set in the elliptic region is well studied in \cite[\S 5.3]{Stefanov2019}, and also see
Section \ref{u_construct}.
The main task of this section is to solve (\ref{dn_eq}) microlocally and study the microlocal polarization of its solution.
\begin{pp}
    In the elliptic region, the DN map $\Lambda$ is a matrix $\Psi$DO with principal symbol
    \begin{equation}\label{symbol_dnmap}
    \sigma_p(\Lambda)=
    \frac{1}{|\xi'|_g^2 - \alpha \beta}
    \begin{pmatrix} 
    \mu (\alpha - \beta)\xi_2^2 + \beta \rho \tau^2 & -\mu\xi_1\xi_2(\alpha-\beta)
    & -i\mu \xi_1 \theta\\
    -\mu\xi_1\xi_2(\alpha-\beta) & \mu (\alpha - \beta)\xi_2^2 + \beta \rho \tau^2 
    & -i\mu \xi_2 \theta\\ 
    i \mu \theta \xi_1 & i \mu \theta \xi_2 & \alpha \rho \tau^2
    \end{pmatrix}
    \equiv \frac{1}{|\xi'|_g^2 - \alpha \beta} N_1
    ,
    \end{equation}
    where
    \begin{align}\label{def_alpha}
    \alpha = \sqrt{|\xi'|_g^2 - c_s^{-2} \tau^2}, 
    \quad \beta = \sqrt{|\xi'|_g^2 - c_p^{-2} \tau^2},
    \quad \theta 
    = 2 |\xi'|_g^2 -  c_s^{-2} \tau^2 - 2 \alpha \beta.
    \end{align}
\end{pp}
\begin{proof}
    In \cite{Stefanov2019} the restriction $U_{out}$ of $U$ to the boundary maps $w|_{x^3=0}$ to $f = u|_{x^3=0}$ and it is an elliptic $\Psi$DO with principal symbol and the parametrix
    \[
    \sigma_p(U_{out}) = 
    \begin{pmatrix} 
    0 & -i\alpha & \xi_1\\
    i\alpha & 0 & \xi_2\\ 
    -\xi_2 & \xi_1 & i\beta
    \end{pmatrix},
    \quad 
    \sigma_p(U_{out}^{-1}) = 
    -\frac{i}{\alpha(|\xi'|_g^2 - \alpha\beta)}
    \begin{pmatrix} 
    -\xi_1\xi_2 & \xi_1^2 - \alpha\beta & -i\alpha\xi_2\\
    -(\xi_2^2 - \alpha\beta) & \xi_1\xi_2 & i\alpha \xi_1\\ 
    i\alpha \xi_1 & i\alpha \xi_2 & -\alpha^2
    \end{pmatrix}.
    \]
    The operator $M_{out}$ that maps $w|_{x^3=0}$ to the Neumann boundary data $l= Nu|_{x^3=0}$ is an $\Psi$DO 
    with principal symbol
    \[
    \sigma_p(M_{out}) = -i 
    \begin{pmatrix} 
    -\mu\xi_1\xi_2 & \mu(\xi_1^2 + \alpha^2) & 2i \mu \beta \xi_1 \\
    -\mu(\xi_2^2 + \alpha^2) & \mu\xi_1\xi_2 & 2i \mu\beta\xi_2\\ 
    -2i \mu\alpha \xi_2 & 2i \mu \alpha \xi_1 & -2\mu|\xi'|_g^2+\rho \tau^2
    \end{pmatrix}.
    \]
    We emphasize that the notation $M_{out}$ we use here is sightly different from that in \cite[\S 6.1]{Stefanov2019}. 
    Then the principal symbol of the DN map 
    \[
    \sigma_p(\Lambda) = \sigma_p(M_{out} U_{out}^{-1}) = \sigma_p(M_{out}) \sigma_p( U_{out}^{-1})
    \]
    can be computed and is given in (\ref{symbol_dnmap}).
\end{proof}
\subsection{Diagonalization of the DN map}
In the following, we first diagonalize $\sigma_p(\Lambda)$ in the sense of matrix diagonalization and then we microlocally decouple the system $\Lambda f = l$ up to smoothing operators. 
By \cite{Stefanov2000}, first we have
\begin{align}\label{V0N1}
V_0^*(t, x', \tau, \xi') N_1(t, x', \tau, \xi') V_0(t, x', \tau, \xi')= 
\begin{pmatrix} 
\beta \rho \tau^2 & -i{{|\xi'|}_g} \mu \theta & 0\\
i{{|\xi'|}_g} \mu \theta & \alpha \rho \tau^2 & 0\\
0 & 0 & \mu \alpha (|\xi'|_g^2 - \alpha \beta),
\end{pmatrix}.
\end{align}
where
\begin{equation*}\label{v0}
V_0(t, x', \tau, \xi') = 
\begin{pmatrix} 
{\xi_1}/{{{|\xi'|}_g} } & 0 & -{\xi_2}/{{{|\xi'|}_g} }\\
{\xi_2}/{{{|\xi'|}_g} } & 0 & {\xi_1}/{{{|\xi'|}_g} }\\ 
0 & 1 & 0
\end{pmatrix}.
\end{equation*}

Then let
\begin{align*}
&m_1(t, x',\tau, \xi') = \frac{(\alpha+\beta)\rho \tau^2 - \sqrt{\varrho}}{2}, 
&m_2(t, x',\tau, \xi') = \frac{(\alpha+\beta)\rho \tau^2 + \sqrt{\varrho}}{2}, \\
&m_3(t, x',\tau, \xi') = \mu \alpha(|\xi'|_g^2 - \alpha\beta),
&\text{ with } \varrho = (\alpha-\beta)^2 \rho^2 \tau^4 + 4 |\xi'|_g^2 \mu^2 \theta^2 >0.
\end{align*}
Notice we always have the following equalities
\begin{equation}\label{m1m2}
m_1 + m_2 = (\alpha + \beta) \rho \tau^2, \qquad
m_1m_2 = \alpha \beta \rho^2 \tau^4 - |\xi'|_g^2 \mu^2 \theta^2.
\end{equation}
We conclude that the principal symbol $\sigma_p(\Lambda)$ can be diagonalized 
\[
W^{-1}(t, x',\tau, \xi') \sigma_p(\Lambda) W(t, x',\tau, \xi') = 
\frac{1}{|\xi'|_g^2 - \alpha \beta}
\begin{pmatrix} 
m_1(t, x',\tau, \xi') & 0 & 0\\
0 & m_2(t, x',\tau, \xi') & 0\\
0 & 0 & m_3(t, x',\tau, \xi')
\end{pmatrix}
\]
by a matrix 
\[
W(t, x', \tau, \xi') = V_0(t, x', \tau, \xi')V_1(t, x', \tau, \xi'), 
\]
where
\begin{equation*}\label{v1}
V_1(t, x',\tau, \xi') = 
\begin{pmatrix} 
{i {{|\xi'|}_g}  \mu \theta}/{k_1} 
& {i {{|\xi'|}_g}  \mu \theta}/{k_2} & 0\\
{(\beta \rho \tau^2 - m_1 )}/{k_1} 
& {(\beta \rho \tau^2 - m_2 )}/{k_2} & 0\\
0 & 0 & 1
\end{pmatrix},
\end{equation*}
with
\begin{align}\label{def_k}
k_j = \sqrt{(\beta \rho \tau^2 - m_j)^2 + |\xi'|_g^2\mu^2 \theta^2}, \quad  \text{for } j=1,2.
\end{align}
More specifically,
\begin{equation}\label{W}
W(t, x',\tau, \xi') = 
\begin{pmatrix} 
{i \mu \theta \xi_1}/{k_1}
&{i \mu \theta \xi_1}/{k_2} & -{\xi_1}/{|\xi|}\\
{i \mu \theta \xi_2}/{k_1}
& {i \mu \theta \xi_2}/{k_2} & {\xi_2}/{|\xi|}\\
{(\beta \rho \tau^2 - m_1)}/{k_1} &  {(\beta \rho \tau^2 - m_2)}/{k_2} & 0
\end{pmatrix}
\end{equation}
is an unitary matrix.
Here $m_1, m_2, m_3$ are the eigenvalues of $N_1(t, x', \tau, \xi')$
smoothly depending on $t, x', \tau, \xi'$.
The eigenvalues $\widetilde{m}_j(t,x', \tau, \xi')$ of $\sigma_{p}(\Lambda)$ are given by 
\[
\widetilde{m}_j(t,x', \tau, \xi') =  m_j(t,x', \tau, \xi') /({|\xi'|_g^2 - \alpha \beta}),
\]
for $j = 1, 2, 3$.
Notice that $m_2, m_3$ are always positive. 
It follows that only $m_1(t, x',\tau, \xi')$ could be zero and this happens
if and only if the determinant of the $2\times 2$ blocks in (\ref{V0N1}) equals zero, i.e.
\begin{align*}
0 = \alpha \beta \rho^2 \tau^4 - |\xi'|_g^2 \mu^2 \theta^2 
&=
\alpha \beta \rho^2 \tau^4 - |\xi'|_g^2 \mu^2 (|\xi'|_g^2 + \alpha^2 - 2\alpha \beta)^2\\
&= (|\xi'|_g^2 - \alpha \beta)\underbrace{(4 \mu^2 \alpha \beta |\xi'|_g^2 - (\rho \tau^2 - 2 \mu |\xi'|_g^2)^2)}_{R(\tau, \xi')}. \stepcounter{equation}\tag{\theequation}\label{Rayleigh_determinant}
\end{align*}
Notice the elliptic region has two disconnected comportments $\pm \tau >0$. 
We consider the analysis for $\tau >0$ and the other case is similar. Define $s = {\tau}/{{{|\xi'|}_g} }$ and let 
\begin{align}\label{def_ab}
&a(s) = \frac{\alpha}{{{|\xi'|}_g} } = \sqrt{1 - c_p^{-2} s^2},
\quad b(s) = \frac{\beta}{{{|\xi'|}_g} } = \sqrt{1 - c_s^{-2} s^2}, \nonumber \\
&\theta(s) = \frac{\theta}{|\xi'|_g^2} = 2  -  c_s^{-2} s^2 - 2 a(s) b(s), 
\quad k_j(s) =  \frac{k_j}{{{|\xi'|}_g}^3}, \quad  \text{for } j=1,2.
\end{align}
Then (\ref{Rayleigh_determinant}) is equivalent to 
\begin{align}\label{Rwave}
R(s) \equiv \frac{R(\tau, \xi')}{|\xi'|_g^4} = 4\mu^2 a(s)b(s) - (\rho s^2- 2\mu)^2 =0. 
\end{align}
It is well-known that at fixed point $(t, x')$, there exists a unique simple zero $s_0$ satisfying $R(s) =0 $ for $0< s < c_s < c_p$. This zero $s_0$ corresponds to a wave called the Rayleigh wave and it is called the \textbf{Rayleigh speed} $c_R \equiv {s_0} < c_s < c_p$. Rayleigh waves are first studied in \cite{Rayleigh1885}.
Since $s_0$ is simple, i.e., $R'(s_0) \neq 0$, 
by the implicit function theorem we have the root of $R(s) =0$ can be written as a smooth function $s_0(t, x')$ near a small neighborhood of the fixed point. 
Then we can write $m_1(t, x', \tau, \xi')$ as a product of $(s-s_0(t, x'))$ and an elliptic factor, i.e.
\begin{align}\label{def_e0}
\widetilde{m}_1(t, x', \tau, \xi') = \frac{1}{{{|\xi'|}_g}  - \alpha\beta} m_1(t,x', \tau, \xi') = e_0(t,x',\tau, \xi') i (\tau - c_R(t,x'){{|\xi'|}_g} ),
\end{align}
where $e_0$ is nonzero and homogeneous in $(\tau, \xi')$ of order zero 
\begin{align*}
e_0(t,x',\tau, \xi') = \frac{m_1(t,x',\tau, \xi')}{ i (|\xi'|_g - \alpha \beta)(\tau - c_R(t,x')|\xi'|_g)} = \frac{R(s)}{ i (s - c_R(t,x'))m_2(s)}.
\end{align*}
The last equality is from (\ref{m1m2}), (\ref{Rayleigh_determinant}), (\ref{Rwave}).
There is a characteristic variety  
\[
\Sigma_R = \{(t,x', \tau, \xi'), \ \tau^2 -  c_R^2(t,x') |\xi'|_g^2 = 0 \}
\]
corresponding to ${m}_1 = 0$.
In particular, by (\ref{m1m2} ) along $\Sigma_R$ we have
\begin{align}\label{e0}
e_0(t,x',\tau, \xi') = \frac{R'(c_R)}{2i(a(c_R)+ b(c_R))\rho c_R^2 }.
\end{align}
In order to fully decouple the system up to smoothing operators, 
we want the three eigenvalues to be distinct. 
Notice this is not necessary in our situation, since with $m_2, m_3>0$ one can always decouple the system into a hyperbolic one and an elliptic system near $s_0$.
\begin{claim}
    Near $s=c_R$, the eigenvalues $m_1(t,x', \tau, \xi'), m_2(t,x', \tau, \xi'), m_3(t,x', \tau, \xi')$ are distinct.
\end{claim}
\begin{proof}
    We already have $m_1 \neq m_2$.
    Additionally, one can show that $m_2 > m_3$ is always true by the following calculation
    \begin{align*}
    m_2 - m_3  
    = & {{|\xi'|}_g}^3 \big(
    \frac{(a+b)\rho s^2 + \sqrt{(a-b)^2 \rho^2 s^4 + 4 \mu^2(1+a^2 -2ab)^2 }}{2}
    - \mu a(1 -ab)
    \big)\\
    > &{{|\xi'|}_g}^3 \big(
    \frac{(a+b)\mu(1-a^2) + 2 \mu(1+a^2 -2ab) }{2}- \mu a(1 - ab) \big)\\
    = &\frac{\mu}{2}\big((b-a) + a^2(b-a) + (a-b)^2 + 1-b^2)>0,
    \end{align*}
    where  $a(s), b(s)$ are defined in (\ref{def_ab}).
    The values of $m_1$ and $m_3$ might coincide but near $\Sigma_R$ they are separate, since $m_1$ is close to zero while $m_3 = \mu \alpha (|\xi'|_g^2 - \alpha \beta)>0$ has a positive lower bound. Therefore, near $\Sigma_R$ we have three distinct eigenvalues.    
\end{proof}

Let $\widetilde{W}(t,x', D_t, D_{x'})$ be an elliptic $\Psi$DO of order zero as constructed in \cite{Taylor2017}
with the principal symbol equal to $W(t,x',\tau, \xi')$. 
Let the operators $e_0(t,x', D_t, D_{x'}) \in \Psi^0$  with symbol 
$e_0(t,x', \tau, \xi')$
and $\widetilde{m}_j(t,x', D_t, D_{x'})\in \Psi^1$ with symbols $\frac{1}{{{|\xi'|}_g}  - \alpha \beta} m_j(t,x', \tau, \xi')$, for $j= 2,3$.
Near some fixed $(t,x', \tau, \xi') \in \Sigma_R$, the DN map $\Lambda$ can be fully decoupled as 
\begin{align*}
\widetilde{W}^{-1} \Lambda \widetilde{W}  =
\begin{pmatrix} 
 e_0 (\partial_t - i c_R(t,x')\sqrt{-\Delta_{x'}}) + r_1 & 0 & 0\\
0 & \widetilde{m}_2 + r_2 & 0\\
0 & 0 & \widetilde{m}_3 + r_3
\end{pmatrix} 
\mod \Psi^{-\infty},
\end{align*}
where 
$r_1(t,x', D_t, D_{x'}), r_2(t,x', D_t, D_{x'}), r_3(t,x', D_t, D_{x'}) \in \Psi^0$
are the lower order term.
If we define 
\begin{align}\label{def_lambda}
r(t, x', D_t, D_x')= e_0^{-1} r_1 \in \Psi^0, \quad \lambda(t, x', D_{x'}) = c_R(t, x') \sqrt{-\Delta_{x'}} \in \Psi^1
\end{align} 
in what follows,
then the first entry in the first row can be written as $e_0(\partial_t - i \lambda(t, x', D_{x'}) + r(t, x', D_t, D_x'))$. 
\begin{remark}
    Each entry of the matrix $\sigma_p(\Lambda)$ is homogeneous in $(\tau, \xi')$ of order $1$
    and that of $W(t,x', \tau, \xi')$ is homogeneous of order $0$.
    The operator $e_0(t, x', D_t, D_x')$ has a homogeneous symbol, which implies
    its parametrix will have a classical one. 
    After the diagonalization of the system,  the operator $r_1(t, x', D_t, D_x')$ have
    a classical symbol, and so does $r(t, x', D_t, D_x')$.
\end{remark}
\begin{remark}
If the density $\rho$ and the Lam\'e parameters $\lambda, \nu$ are time-dependent, then $\lambda, r$ depend on $t, x$. Otherwise, the eigenvalues $m_1, m_2, m_3$ only depends on $x, \xi, \tau$, and therefore we have $s_0(x), c_R(x)$ and $\lambda(x, \xi), r(x, D_t, D_x)$ instead of the functions and operators above.
\end{remark}
Now 
let 
\begin{align}\label{def_fl}
&h =
\begin{pmatrix}
h_1\\
h_2\\
h_3
\end{pmatrix} 
=
\widetilde{W}^{-1}
\begin{pmatrix}
f_1\\
f_2\\
f_3
\end{pmatrix} 
=  \widetilde{W}^{-1} f,
&\tilde{l} 
=
\begin{pmatrix}
\tilde{l}_1\\
\tilde{l}_2\\
\tilde{l}_3
\end{pmatrix} 
=\widetilde{W}^{-1}
\begin{pmatrix}
{l}_1\\
{l}_2\\
{l}_3
\end{pmatrix} 
= \widetilde{W}^{-1}l,
\end{align} 
where $u_j$ is the component of any vector valued distribution $u$ for $j = 1, 2, 3$.
Solving $\Lambda f = l \mod C^\infty$ is microlocally equivalent to solving the following system
\begin{equation}\label{hyperbolic_eq}
\begin{cases}
(\partial_t - i c_R(t,x')\sqrt{-\Delta_{x'}} + r(t, x', D_t, D_x')) h_1 = e_0^{-1}\tilde{l}_1, &\mod C^\infty,\\
(\widetilde{m}_2 + r_2)h_2 = \tilde{l}_2,  &\mod C^\infty,\\
(\widetilde{m}_3 + r_3)h_3 = \tilde{l}_3, &\mod C^\infty.
\end{cases}
\end{equation}
In the last two equations, the operators $\widetilde{m}_j + r_j$ are elliptic so we have $h_j = (\widetilde{m}_j + r_j)^{-1} \tilde{l}_j \mod C^\infty$, for $j=2,3$.
The first equation is a first-order hyperbolic equation with lower order term.

\subsection{Inhomogeneous hyperbolic equation of first order}
For convenience, in this subsection we use the notation $x$ instead of $x'$. Suppose $x \in \mathbb{R}^n$.
\begin{df}
    Let $\lambda(t,x, D_{x}) \in \Psi^1$ be an elliptic operator with a classical symbol smoothly depending on a parameter $t$ and 
    the lower term $r(t, x, D_t, D_{x}) \in \Psi^0$ with a classical symbol.
\end{df}
In this subsection we are solving the inhomogeneous hyperbolic equation
\begin{equation}\label{iheq_r}
\begin{cases}
(\partial_t - i \lambda(t,x, D_{x} + r(t,x,D_t, D_x))) w = g(t,x), \quad t >0 \\
w(0,x) = 0.
\end{cases}
\end{equation}
where $g(t,x) \in \mathcal{E}'((0, T) \times \mathbb{R}^n)$ with microsupport in the elliptic region.

Generally, the operator $\partial_t - i \lambda(t,x, D_{x})$ is not a $\Psi$DO unless the principal symbol of $\lambda$ is smooth in $\xi$ at $\xi = 0$. 
However, since we only consider the elliptic region, we can always multiply it by a cutoff $\Psi$DO whose microsupport is away from $\xi = 0$ and this gives us a $\Psi$DO.  
Therefore, by the theorem of propagation of singularities by H\"{o}rmander, we have
$
\text{\( \WF \)} (w) \subset \text{\( \WF \)} (g) \cup C_F \circ \text{\( \WF \)} (g)
$
if $w$ is the solution to (\ref{iheq_r}),
where $C_F$ is given by the flow of $H_{\partial_t - i \lambda(t,x, D_{x})}$, for a more explicit form see (\ref{C1}).
\subsubsection{Homogeneous equations}\label{subsec_he}
We claim the homogeneous first-order hyperbolic equation with lower terms given an initial condition
\begin{equation}\label{home_hyperbolic}
\begin{cases}
(\partial_t - i \lambda(t,x, D_{x}) + r(t, x, D_t, D_x)) v = 0,  \mod C^\infty\\
v(0,x) = v_0(x)\in \mathcal{E}'(\mathbb{R}^n),
\end{cases}
\end{equation}
has a microlocal solution by the geometric optics construction
\begin{equation}\label{goc1}
v(t, x)  = \int a(t,x,\xi) e^{i\varphi(t, x, \xi)} \hat{v}_0(\xi) \diff \xi, \quad \mod C^\infty,
\end{equation}
where we require $a(t,x, \xi) \in S^0$ and $\varphi(t,x,\xi)$ is a phase function that is smooth, real valued, homogeneous of order one in $\xi$ with $\nabla_x \varphi \neq 0$ on the conic support of $a$. These assumptions guarantees the oscillatory integral (\ref{goc1}) is a well-defined Lagrangian distribution. 
The procedure presented in the following is based on the construction in  \cite[VIII.3]{Taylor2017}.

If we suppose
\[
(\partial_t - i \lambda(t,x, D_{x}) + r(t, x, D_t, D_x))v(t, x)  = \int c(t,x,\xi) e^{i\varphi(t, x, \xi)} \hat{v}_0(\xi) \diff \xi,
\]
then we have
\[
c(t, x, \xi) = i \varphi_t a + \partial_t a - ib + d,
\]
where
\[
b = e^{-i\varphi} \lambda (a e^{i\varphi}), \quad d = e^{-i\varphi} r (a e^{i\varphi})
\]
have the asymptotic expansions according to the Fundamental Lemma. In the following we use the version given in \cite{Treves1980}.
\begin{lm}\cite[Theorem 3.1]{Treves1980}\label{fundl}
    Suppose $\phi(x, \theta)$ is a smooth, real-valued function for $x \in \Omega, \theta \in S^n$ and the gradient $\nabla_x \phi \neq 0$. 
    Suppose $P(x, D_x)$ is a pseudodifferential operator of order $m$. 
    Write 
    $
    \phi(x, \theta) - \phi(y, \theta) = (x-y)\cdot \nabla \phi(y, \theta) - \phi_2(x,y),
    $
    where $\phi_2(x, y) = \mathcal{O}(|x-y|^2)$. Then we have the asymptotic expansion
    \begin{align}
    e^{-i\rho \phi} P(x, D_x) (a(x) e^{i\rho \phi}) \sim \sum_\alpha \frac{1}{\alpha!} \partial_\xi^\alpha P(x, \rho\nabla_x \phi) \mathcal{R}_\alpha (\phi; \rho, D_x) a, \text{ for }\rho >0,
    \end{align}
    where we use the notation 
    \begin{align}
    \mathcal{R}_\alpha (\phi; \rho, D_x) a = D_y^\alpha\{e^{i\rho \phi_2(x,y)} a(y)\}|_{y=x}.
    \end{align}
\end{lm}
\begin{remark}
    Indeed, we can write 
    \[
    \phi(x, \theta) - \phi(y, \theta) = (x-y) \cdot \nabla \phi(y,\theta) + (x-y)^T
    \big(\int_0^1 \int_0^1 (t\nabla_y^2 \phi(y+st(x-y), \theta) ) \diff s \diff t 
    \big) (x-y),
    \]
    which implies $-\phi_2(x,y)$ equals to the last term above.
    One can show $\mathcal{R}_\alpha (\phi; \rho, D_x) a$ is a polynomial w.r.t. $\rho$ with degree $\leq \left \lfloor{|\alpha|/2}\right \rfloor$.
    In particular, we have
    \begin{align*}
    &\mathcal{R}_\alpha (\phi; \rho, D_x) a = D_x^\alpha a(x), \quad  \text{ for } |\alpha| =1,\\
    &\mathcal{R}_\alpha (\phi; \rho, D_x) a  = D_x^\alpha a(x)  - i \rho D_x^\alpha \phi(x, \theta) a(x), \quad  \text{ for } |\alpha| =2.
    \end{align*}
\end{remark}
From Lemma \ref{fundl}, we have the following asymptotic expansions of $b, d$ by writing $(\tau, \xi) = \rho \theta, \varphi = \rho \phi(x, \theta)$, with $\rho >0, \theta \in S^{n+1}$. 
We use the notation $\lambda^{(\alpha)} = \partial_\xi^\alpha$ and $r^{(\alpha)} =  \partial_{(\tau, \xi)}^\alpha r$ to have
\begin{align*}
b &\sim 
\sum_\alpha \frac{1}{\alpha!} \lambda^{(\alpha)}(t,x,\rho \nabla_x \phi) \mathcal{R}_\alpha (\phi; \rho, D_x) a(t,x,\rho \theta)\\
&\sim  \underbrace{ \lambda(t,x,\rho \nabla_x \phi) a }_{\text{\color{blue} order $\leq 1$ }} +  \sum_{|\alpha|=1} \underbrace{ \lambda^{(\alpha)} (t,x,\rho \nabla_ \phi) D_x^\alpha a}_{\text{\color{blue} order $\leq 0$ }}\\
&+\sum_{|\alpha|=2}\frac{1}{\alpha!}(\underbrace{ \lambda^{(\alpha)}(t,x,\rho \nabla_x \phi)D_x^\alpha  a }_{\text{\color{blue} order $\leq -1$}}
- i \underbrace{  \lambda^{(\alpha)} (t,x,\rho \nabla_x \phi) (D_x^\alpha \varphi) a }_{\text{\color{blue} order $\leq 0$}}) + \underbrace{\ldots}_{\text{\color{blue} order $\leq -1$}},
\end{align*}
and
\begin{align*}
d & \sim  \sum_{\alpha} \frac{1}{\alpha!} r^{(\alpha)}(t,x,\rho \nabla_{t,x} \phi)
\mathcal{R}_\alpha (\phi; \rho, D_{t,x}) a(t,x,\rho \theta)\\
&\sim  \underbrace{ r(t,x,\rho \nabla_{t,x} \phi) a }_{\text{\color{blue} order $\leq 0$ }} +  \sum_{|\alpha|=1} \underbrace{ r^{(\alpha)}(t,x,\rho \nabla_{t,x} \phi) D_{t,x}^\alpha a}_{\text{\color{blue} order $\leq -1$ }}
+ \sum_{|\alpha|=2} \frac{1}{\alpha!} ( \underbrace{r^{(\alpha)}(t,x,\rho \nabla_{t,x} \phi)
    D_{t,x}^\alpha a}_{\text{\color{blue} order $\leq -2$}} \\
&- i \underbrace{ r^{(\alpha)}(t,x,\rho \nabla_{t,x} \phi) (D_{t,x}^\alpha \varphi) a)}_{\text{\color{blue} order $\leq -1$}} +
\underbrace{\ldots}_{\text{\color{blue} order $\leq -2$}}.
\end{align*}
Indeed for each fixed $\alpha$, the order of each term in the asymptotic expansion of $b$ is no more than $1-|\alpha| + \left \lfloor{|\alpha|/2}\right \rfloor$, that of $d$ is no more than $0-|\alpha| + \left \lfloor{|\alpha|/2}\right \rfloor$. 

To construct the microlocal solution, we are finding proper $a(t,x,\xi)$ in form of $\sum_{j \leq 0} a_j(t,x,\xi)$, where $a_j \in S^{-j}$ is homogeneous in $\xi$ of degree $-j$. We also write $b, c, d$ as asymptotic expansion such that
\[
c(t,x,\xi) \sim \sum_{j \leq 1} c_j(t,x,\xi) = (i \varphi_t  - i\lambda(t,x,\nabla_x \varphi) ) a + \sum_{j \leq 0} (\partial_t a_j - i b_j + d_j ),
\]
where we separate the term of order $1$ since it gives us the eikonal equation 
\begin{align}\label{eikonal}
c_1 = i (\varphi_t   - \lambda_1(t,x,\nabla_x \varphi))a = 0 \quad \Rightarrow \quad \varphi_t   = \lambda_1(t,x,\nabla_x \varphi), \qquad \text{ with } \varphi(0, x, \xi) = x \cdot \xi
\end{align}
for the phase function. 
Then equating the zero order terms in $\xi$ we have
\begin{align}\label{transport_h_a0}
\partial_t a_0 - ib_0 + d_0 =c_0 = 0  
\Rightarrow  X a_0 - \gamma a_0 = 0,
\qquad \text{ with } a_0(0,x,\xi) = 1,
\end{align}
where we set 
$X = \partial_t - \nabla_\xi \lambda_1 \cdot \nabla_x$ be the vector field
and 
\begin{align*}
\gamma = (i \lambda_0(t,x,\nabla_x \varphi ) +  \sum_{|\alpha|=2} \frac{1}{\alpha!} \lambda_1^{(\alpha)} D_x^\alpha \varphi  -r_0(t,x, \lambda_1,\nabla_x \varphi ).
\end{align*}
Then for lower order terms, i.e. $j \leq -1$, we have
\begin{align}\label{transport_h_aj}
0 = c_j = Xa_j - \gamma a_j  - e_j,
\qquad \text{ with } a_j(0,x,\xi) = 0,
\end{align}
where $e_j$ is expressible in terms of $\varphi, a_0, a_{-1}, \ldots, a_{j-1}, \lambda, r$. 
This finishes the construction in (\ref{goc1}).

\begin{remark}\label{rm_heq}
This construction of microlocal solution is valid in a small neighborhood of $t=0$,
since the Eikonal equation is locally solvable. 
However, we can find some $t_0 >0$ such that the solution $v$ is defined and use the value at $t_0$ as the Cauchy data to construct a new solution for $t>t_0$, for the same arguments see \cite[\S 3.1]{Stefanov2019}
\end{remark}
\subsubsection{Inhomogeneous equations when $r(t,x, D_t, D_x) = 0$.} 
Now we are going to solve the inhomogeneous equation with zero initial condition. 
A simpler case would be when the lower order term $r(t, x, D_t, D_x)$ vanishes, i.e.
\begin{equation}\label{iheq}
\begin{cases}
(\partial_t - i \lambda(t,x, D_{x})) w = g(t,x), \quad t >0 \\
w(0,x) = 0.
\end{cases}
\end{equation}
In this way the microlocal solution can be obtained by the Duhamel's principle.
Indeed, let the phase function $\varphi(t,x,\xi)$, the amplitude $a(t,x,\xi)$ to be constructed for 
solutions to the homogeneous first order hyperbolic equation with an initial condition as in (\ref{home_hyperbolic}).
More specifically, suppose the phase $\varphi(t,x,\xi)$ solves the eikonal equation (\ref{eikonal})
and the amplitude $a(t,x,\xi)  = a_0 + \sum_{j \leq -1}a_j$ solves the transport equation (\ref{transport_h_a0}) and (\ref{transport_h_aj}) with
$\gamma = (i \lambda_0 + \sum_{|\alpha|=2} \frac{1}{\alpha!} \lambda_1^{(\alpha)} D_x^\alpha \varphi)$.
Then the solution to (\ref{iheq}) up to a smooth error is given by 
\begin{align}\label{solution_r0}
w(t,x) 
= \int H(t-s) a(t-s, x, \xi) e^{i(\varphi(t-s, x, \xi) - y \cdot \xi)} l(s,y)  \diff y \diff \xi \diff s \equiv  L_{\varphi,a}l(t, x),
\end{align}
where we define $L_{\varphi, a}$ as the solution operator to (\ref{iheq}) with the phase $\varphi$ and the amplitude $a$. 
Here, the kernel of $L_{\varphi, a}$ 
\begin{align*}
k_L(t,x,s,y) =  H(t-s)\int  a(t-s, x, \xi) e^{i(\varphi(t-s, x, \xi) - y \cdot \xi)} l(s,y) \diff \xi 
\end{align*}
can be formally regarded as the product 
\begin{align}\label{def_kF}
k_L(t,x,s,y) =  H(t-s) k_F(t,x,s,y) 
\end{align}
of a conormal distribution $H(t-s)$ and a Lagrangian distribution 
where $k_F(t,x,s,y)$ is analyzed by the following claim.
\begin{claim}\label{op_F}
    The kernel $k_F$ defined by (\ref{def_kF}) is a Lagrangian distribution associated with the canonical relation 
    \begin{align}\label{C1}
    C_F = \{
    (
    \underbrace{
        t, x, 
        \varphi_t(t-s, x,\xi), \varphi_x(t-s, x,\xi),
        s, y,
        \varphi_t(t-s, x,\xi), \xi}_{t,x,\hat{t}, \hat{x}, s, y, \hat{s}, \hat{y}}
    ), \text{ with } y = \varphi_\xi(t-s, x, \xi)
    \}.
    \end{align}
    It is the kernel of an FIO $F_{\varphi,a}$ of order $-\frac{1}{2}$.
\end{claim}
\begin{remark}
    In Euclidean case, we have $\varphi(t,x, \xi) = t |\xi| + x \cdot \xi$. Then
    \[
    C_F = \{
    (t, x, |\xi|,\xi, s, y, 
    |\xi|,\xi), \text{ with } x = y - (t-s) \frac{\xi}{|\xi|}
    \}.
    \]
\end{remark}
Further, we show that $k_L$ is a distribution kernel such that the microlocal solution (\ref{solution_r0}) is well-defined for any $g(t,x) \in \mathcal{E}'((0, T) \times \mathbb{R}^n)$ supported in the elliptic region.
\begin{pp}
    The kernel $k_L$ is a well-defined distribution, 
    with the twisted wave front set satisfying
    $\text{\( \WF \)}' (k_L) \subset C_0 \cup C_\Delta \cup C_F$,
    where
    \begin{align*}
    &C_0 = \{
    (t,x,\hat{t}, \hat{x}, s,y, \hat{s}, \hat{y}) = 
    (
    {
        t,x, 
        \mu,0,
        t, y, 
        \mu,0
    }
    ), \ 
    \mu \neq 0
    \}, \\
    & C_\Delta = \{
    (t,x,\hat{t}, \hat{x}, s,y, \hat{s}, \hat{y}) = 
    (
    {t,x, \tau, \xi, t, x, \tau, \xi}
    ), \ (\tau, \xi) \neq 0
    \}.
    \end{align*}
    If $g(t,x)  \in \mathcal{E}'((0,T) \times \mathbb{R}^n)$ microlocally supported in the elliptic region, then $L_{\varphi,a}$ is a distribution with 
    $
    \text{\( \WF \)} (L_{\varphi,a}g) \subset \text{\( \WF \)} (g) \cup C_F \circ \text{\( \WF \)} (g),
    $
    where $C_F$ is defined in (\ref{C1}).
\end{pp}
\begin{proof}
    We use \cite[Theorem 8.2.10]{Hoermander2003}, with the assumption that the principal symbol of $\lambda(t,x, D_x)$ is homogeneous in $\xi$ of order one.
    By \cite[Theorem 2.5.14]{Hoermander1971}, for $ g(t, x) \in \mathcal{E}'((0,T) \times \mathbb{R}^n)$, since $\text{\( \WF \)} (k_L)$ has no zero sections, we have
    \[
    \text{\( \WF \)} (L_{\varphi,a} g) \subset C_0 \circ  \text{\( \WF \)}(g) \cup \text{\( \WF \)}(g) \cup  C_F \circ \text{\( \WF \)}(g).
    \]
    In particular, the first term in the right side is ignorable in general. 
    However, if we assume there is no $(t, x, \mu, 0) \in \text{\( \WF \)} (g)$, which is true if $g$ is microlocally supported in the elliptic region,
    then for $w$ satisfying equation (\ref{iheq}), 
    the wave front set
    $
    \text{\( \WF \)} (w) \subset \text{\( \WF \)} (g) \cup C_F \circ \text{\( \WF \)} (g). 
    $
    Especially for $t \geq T$, we have
    $ \text{\( \WF \)} (w) \subset
    C_F \circ \text{\( \WF \)} (g). $
\end{proof}

\subsubsection{Inhomogeneous equations with nonzero $r(t,x, D_t, D_x)$.} 
When the lower order term $r(t, x, D_t, D_x)$ is nonzero, the Duhamel's principle does not work any more. Instead, we can use the same iterative procedure as in \cite[Section 5]{Treves1980} to construct an operator $\tilde{e}(t,x, D_t, D_x) \in \Psi^{-1}$ such that  
\begin{align}\label{etilde}
&\partial_t - i \lambda(t,x, D_{x}) + r(t, x, D_t, D_x) \nonumber\\
=& (I - \tilde{e})(\partial_t - i \lambda(t,x, D_{x}) + \sum_{j \leq 0} \tilde{r}_j(t, x, D_x)) \mod \Psi^{-\infty}.
\end{align}
%
%
Here each $\tilde{r}_j(t, x, D_x)$ has a classical symbol so does their sum. 
In particular, the principal symbol of $\tilde{r}_0(t,x,D_x)$ is $r_0(t, x, \lambda_1(t,x,\xi), \xi)$.
The similar trick is performed for $\lambda$-pseudodifferential operators in  \cite{Stefanov2000}.

In this way, the microlocal solution to the inhomogeneous hyperbolic equation can be written as
\begin{align}
w(t,x) 
&= L_{\varphi, a} (I - \tilde{e})^{-1} g  \label{def_L} \\
&= \int H(t-s) a(t-s, x, \xi) e^{i(\varphi(t-s, x, \xi) - y \cdot \xi)} ((I - \tilde{e})^{-1}g)(s,y)  \diff y \diff \xi \diff s \label{solution_ih_r0}
\end{align}
where $L_{\varphi,a}$ is the solution operator of the inhomogeneous first order hyperbolic equation $(\partial_t - i \lambda(t,x, D_{x}) + \sum_{j=0} \tilde{r}_j(t, x, D_x)) v = g$ with zero initial condition.
Since $(I - \tilde{e})$ is an elliptic $\Psi$DO with principal symbol equal to $1$, we have the same conclusion for the wave front set of $w$ as the simpler case. 
\begin{pp}\label{prop_solution_ih}
    Assume $g(t,x') \in \mathcal{E}'((0,T) \times \mathbb{R}^2)$ microlocally supported in the elliptic region.
    Then the inhomogeneous first-order hyperbolic equation (\ref{iheq_r})    
    admits a unique microlocal solution given by (\ref{solution_ih_r0}),
    where the phase function $\varphi(t,x,\xi)$ and the amplitude $a(t,x,\xi)$ are constructed for the operator $(\partial_t - i \lambda(t,x, D_{x}) + \sum_{j=0} \tilde{r}_j(t, x, D_x))$ in (\ref{etilde}), as in Subsection \ref{subsec_he}. 
    More specifically, the phase $\varphi(t,x,\xi)$ solves the eikonal equation (\ref{eikonal});
    the amplitude $a(t,x,\xi)  = a_0 + \sum_{j \leq -1}a_j$ solves the transport equation (\ref{transport_h_a0}) and (\ref{transport_h_aj}) with
    $\gamma  = (i\lambda_0(t,x,\xi) +  \sum_{|\alpha|=2} \frac{1}{\alpha!} \lambda_1^{(\alpha)} D_x^\alpha \varphi -r_0(t,x,\lambda_1, \xi))$.
\end{pp}
\begin{proof}
    To justify the parametrix, we still need to show that if $w$ is the solution to (\ref{iheq_r}) with $g  \in C^\infty$ and $w(0,x) \in C^\infty$,   
    then $w \in C^\infty$ as well. 
    By (\ref{etilde}) it suffices to show this is true when the lower order term $r(t,x,D_t,D_x)$ can be reduced to the form $r(t,x,D_x)$ or vanishes. One can verify that the operator $\partial_t - i \lambda(t,x, D_{x}) + r(t, x, D_x)$ is symmetric hyperbolic as is defined in \cite{Taylor2017}. Then following the same arguments there, by a standard hyperbolic estimates, one can show $w$ is smooth.
\end{proof}
\subsection{The Cauchy problem and the polarization}\label{section_R_cauchy}
In this subsection, before assuming the source $l \in \mathcal{E}'((0,T) \times \Gamma) $ and solving the inhomogeneous equation (\ref{dn_eq}) with zero initial condition, we first assume that the source exists for a limited time for $t<0$ and we have the Cauchy data $f|_{t=0}$ at $t=0$, i.e.
\begin{equation}\label{dn_cauchy}
\Lambda f = 0, \text{ for } t > 0, \quad 
f|_{t=0} \text{ given }.
\end{equation}
Recall the diagonalization of $\Lambda$ in (\ref{def_fl}, \ref{hyperbolic_eq}). The homogeneous equation $\Lambda f=0$ implies 
\begin{align}\label{eg_fh}
f = \widetilde{W} h
= \begin{pmatrix}
\widetilde{W}_{11}\\
\widetilde{W}_{21}\\
\widetilde{W}_{31}\\
\end{pmatrix} h_1 \mod C^\infty
\Leftarrow h_2 = h_3 =0 \mod C^\infty,
\end{align}
where $h_1$ solves the homogeous first-order hyperbolic equation in (\ref{hyperbolic_eq}).
Notice in this case the hyperbolic operator is $\partial_t - i c_R(t,x'){{|\xi'|}_g}  + r(t, x', D_t, D_{x'})$ with $r$ given by (\ref{def_lambda}).
If we have the initial condition $h_{1,0}\equiv h_1|_{t=0}$, then by the construction in Subsection \ref{subsec_he}, then
\begin{align}\label{eg_h1}
h_1(t,x') = \int a(t,x',\xi') e^{i\varphi(t,x',\xi')}  \hat{h}_{1,0}(\xi') \diff \xi' \mod C^\infty,
\end{align}
where the phase function $\varphi$ solves the eikonal equation (\ref{eikonal});
the amplitude $a  = a_0 + \sum_{j \leq -1}a_j$ solves the transport equation (\ref{transport_h_a0}) and (\ref{transport_h_aj}) with
$
\gamma  = ( \sum_{|\alpha|=2} \big(\frac{1}{\alpha!} c_R(t,x')D_{\xi'}^\alpha{{|\xi'|}_g}  D_{x'}^\alpha \varphi-r_0(t,x',\lambda_1, \nabla_{x'}\varphi)\big).
$

To find out how the initial condition of $f$ is related to that of $h$, we plug (\ref{eg_h1}) into (\ref{eg_fh}), use the Fundamental Lemma in Lemma \ref{fundl}, and set $t=0$. 
Since $\varphi(0,x',\xi') = x' \cdot \xi'$, after these steps we get three $\Psi$DOs of order zero, of which the symbols can be computed from the Fundamental Lemma, such that
\begin{align}\label{compatible}
f|_{t=0} = 
\begin{pmatrix}
\widetilde{W}_{11,0}\\
\widetilde{W}_{21,0}\\
\widetilde{W}_{31,0}\\
\end{pmatrix} 
h_{1,0} \mod C^\infty.
\end{align}
In particular, the principal symbols are 
\begin{align*}
&\sigma_p(\widetilde{W}_{11,0}) = \sigma_p({W}_{11}) (0, x', c_R{{|\xi'|}_g} , \xi'), \qquad
\sigma_p(\widetilde{W}_{21,0}) = \sigma_p({W}_{21})(0, x', c_R{{|\xi'|}_g} , \xi'),\\
&\sigma_p(\widetilde{W}_{31,0})= \sigma_p({W}_{31})(0, x', c_R{{|\xi'|}_g} , \xi') 
\end{align*}
and by (\ref{W}) they are elliptic $\Psi$DOs. 
This indicates not any arbitrary initial conditions can be imposed for (\ref{dn_cauchy}). 
Instead, to have a compatible system, we require that there exists some distribution $h_0$ such that
$f(0,x')$ can be written in form of (\ref{compatible}). 
\begin{thm}\label{thm_cauchy}
    Suppose $f(0,x')$ satisfies (\ref{compatible}) with some $h_{1,0} \in \mathcal{E}'$. Then microlocally the homogeneous problem with Cauchy data (\ref{dn_cauchy}) admits a unique microlocal solution
    \begin{align}\label{solution_cauchy}
    f &= \begin{pmatrix}
    \widetilde{W}_{11}\\
    \widetilde{W}_{21}\\
    \widetilde{W}_{31}\\
    \end{pmatrix}
    \int a(t,x',\xi') e^{i\varphi(t,x',\xi')}  \hat{h}_{1,0}(\xi') \diff \xi' \mod C^\infty \nonumber \\
    &=  \int
    \underbrace{ 
    \begin{pmatrix} 
    i \mu \theta(c_R) \nabla_{x_1} \varphi /{{|\nabla_{x'} \varphi |}_g}  \\
    i \mu \theta(c_R)\nabla_{x_2} \varphi /{{|\nabla_{x'} \varphi |}_g} \\
    b(c_R)\rho c_R^2(t,x') 
    \end{pmatrix}
    \frac{a_0(t,x',\xi') e^{i\varphi(t,x',\xi')}}{k_1(c_R)}
    }_{\mathcal{P}}
    \hat{h}_{1,0}(\xi') \diff \xi'
    + \text{ lower order terms},
    \end{align}
    where $c_R$ is the Rayleigh speed, $b(s), \theta(s), k_1(s)$ are defined in (\ref{def_ab}), and $a(t,x',\xi')$ is the amplitude from the geometric optics construction with the highest order term $a_0(t,x',\xi')$.
\end{thm}
With $f$, one can construct the real displacement $u$ as an evanescent mode. 
Notice this theorem gives us a local representation of $f$ in the sense of Remark \ref{rm_heq}.

To understand the polarization of microlocal solution in the theorem above, if we write the real and imaginary part of the term $\mathcal{P}$ separately, then we have
\begin{align*}
&\Re(\mathcal{P})
= 
\frac{|a_0(t,x',\xi')|}{k_1(c_R)}
\begin{pmatrix} 
\mu \theta(c_R) \cos(\varphi + \upsilon + \pi/2) \nabla_{x_1} \varphi /{{|\nabla_{x'} \varphi |}_g}   \\
\mu \theta(c_R) \cos (\varphi + \upsilon + \pi/2) \nabla_{x_2} \varphi /{{|\nabla_{x'} \varphi |}_g} \\
b(c_R)\rho c_R^2(t,x') \sin(\varphi + \upsilon + \pi/2)
\end{pmatrix},\\
&\Im(\mathcal{P})
=
\frac{|a_0(t,x',\xi')|}{k_1(c_R)}
\begin{pmatrix} 
 \mu \theta(c_R) \cos (\varphi + \upsilon)\nabla_{x_1} \varphi /{{|\nabla_{x'} \varphi |}_g}  \\
 \mu \theta(c_R) \cos (\varphi + \upsilon) \nabla_{x_2} \varphi /{{|\nabla_{x'} \varphi |}_g} \\
b(c_R)\rho c_R^2(t,x') \sin (\varphi + \upsilon)
\end{pmatrix},
\end{align*}
where we assume $\upsilon = \arg(a_0(t,x',\xi'))$. 
It follows that the solution in (\ref{solution_cauchy}) can be regarded as a superposition of $\Re(\mathcal{P})$ and $\Im(\mathcal{P})$. We are going to show that each of them has a \textbf{retrograde elliptical motion} in the following sense.

Take the real part $\Re(\mathcal{P})$ as an example, by assuming $\hat{h}_{1,0}(\xi')$ is real. Then the real part of the displacement $f(t,x')$ on the boundary is the superposition of 
$
\Re(\mathcal{P})
\equiv
(
p_1,
p_2,
p_3
)^T
$.
On the one hand, we have the following equation satisfied by the components
\begin{align*}
\frac{p_1^2}{|\mu \theta(c_R)|^2} +  \frac{p_2^2}{|\mu\theta(c_R)|^2} + \frac{p_3^2}{|b(c_R)\rho c_R^2 |^2} = \frac{|a_0(t,x',\xi')|^2}{k_1^2(c_R)} 
\end{align*}
which describes an ellipsoid.

On the other hand, the rotational motion of the particle has the direction opposite to its orbital motion, i.e. the direction of the propagation of singularities. 
Indeed, if we consider the Euclidean case, 
we have $a_0(t,x', \xi')  = 1 $ and $\varphi(t,x',\xi') = t|\xi'| + x' \cdot \xi'$.
It follows that $\upsilon = 0$.
Locally the singularities propagate 
along the path 
$
 -\frac{\xi'}{|\xi|} t + y'
$
in the direction of $(-\xi', 0)$ on the boundary $x^3 = 0$, where $y'$ is the initial point.
W.O.L.G. assume $\xi^1, \xi^2 >0$ in what follows.
In this case, we have
\begin{align*}
\Re(\mathcal{P})
=\frac{1}{k_1(c_R)}
\begin{pmatrix} 
\mu \theta(c_R)  \cos (t|\xi'| + x' \cdot \xi' + \pi/2)  \xi_1/{{|\xi'|}}   \\
\mu \theta(c_R)  \cos (t|\xi'| + x' \cdot \xi' + \pi/2)  \xi_2 /{{|\xi'|}} \\
b(c_R)\rho c_R^2  \sin (t|\xi'| + x' \cdot \xi' + \pi/2) 
\end{pmatrix}.
\end{align*}   
Notice $t|\xi'| + x' \cdot \xi' + \pi/2$ is the angle between the vector $(p_1, p_2, p_3)^T$, as the direction of strongest singularity of the solution, and the plane $x^3=0$ as in Figure \ref{fig_po}.
At each fixed point, the polarized vector $(p_1, p_2, p_3)^T$ rotates in the direction that the angle increases as $t$ increases, i.e. the clockwise direction, while the singularities propagates in the counterclockwise direction.
Therefore, we have a retrograde motion. 
\begin{figure}[h]
    \centering
    \includegraphics[height=0.2\textwidth]{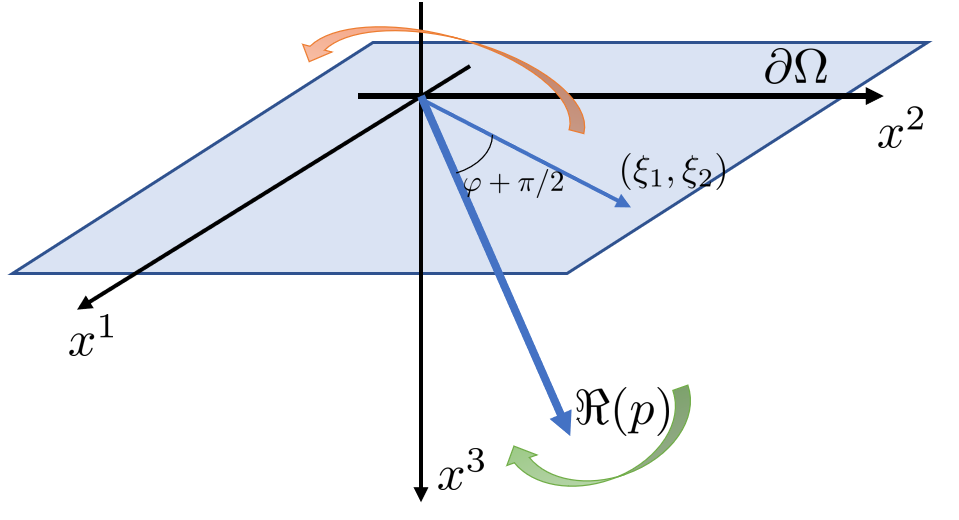}
    \caption{Propagation of the wave and the rotation of the polarization.} \label{fig_po}
\end{figure}


In the general case, notice by (\ref{transport_h_a0}) the leading amplitude satisfies
\begin{align*}
X a_0 - \gamma a_0 = 0,
\qquad \text{ with } a_0(0,x',\xi') = 1,
\end{align*}
where 
\begin{align*}
X = \partial_t - \nabla_{\xi'} \lambda_1 \cdot \nabla_{x'}, \quad \gamma = (\sum_{|\alpha|=2} \frac{1}{\alpha!} \lambda_1^{(\alpha)} D_{x'}^\alpha \varphi  -r_0(t,{x'}, \lambda_1,\nabla_{x'} \varphi ),
\end{align*}
with $\lambda_1 = c_R|\xi'|_g$. 
In the following assume the Rayleigh speed $c_R$ only depends on the space variable $x'$ and we write $\tilde{g} = c^{-2}_R g$. 
Let 
\[
\gamma = \gamma_1(t,x',\xi')  + \gamma_2 (t,x',\xi') i
\] be the decomposition of the real and imaginary part. 
We emphasize that the imaginary part $\gamma_2(t,x',\xi')$ comes from that of $r_0(t,{x'}, \lambda_1,\nabla_{x'} \varphi )$, the lower order term in the decoupled system, and it is a classical symbol of order zero.

The integral curve of $X$ is the unit speed geodesic $(t, \sigma_{y',-\xi'/|\xi'|_{\tilde{g}}} (t))$ and for simplification we omit the initial point $y'$ and the initial direction $-\xi'/|\xi'|_{\tilde{g}}$ to write it as $(t, x'(t))$ for the moment.
Then the solution of the transport equation is 
\begin{align*}
a_0(t,x'(t),\xi') = e^{-\int_0^t \gamma (s, x'(s), \xi') \diff s} = e^{-\int_0^t \gamma_1 (s, x'(s), \xi') \diff s} \cdot e^{-i \int_0^t \gamma_2 (s, x'(s), \xi') \diff s},
\end{align*}
which implies $\upsilon = -\int_0^t \gamma_2 (s, x'(s), \xi') \diff s$ and $\upsilon$ is a classical symbol of order zero.
To find out at each fixed point $x'$ how the polarization  rotates as the time changes, we compute the time derivative 
\begin{align}\label{angle_time}
(\varphi + \upsilon)_t = |\nabla_{x'} \varphi(t, x',\xi')|_{\tilde{g}} + \upsilon_t,
\end{align}
which is positive for $\xi'$ large enough since $ \upsilon_t$ is a symbol of order zero. 
This indicates for large $\xi'$ the angle $\varphi + \upsilon$ increases as $t$ increases and the rotation of the polarization vector is still clockwise, if we assume at the fixed point $\nabla_{x'} \varphi $ points in the direction that $x^1, x^2$ increase.

In this case, the singularities propagates along the null characteristics of $\partial_t - \lambda_1$.
Particularly, the wave propagates along the geodesics $\sigma_{y',-\xi'/|\xi'|_{\tilde{g}}} (t)$, 
where $y'$ is the initial point. 
By Remark \ref{solve_eikonal}, along the geodesics we have $\nabla_{x'} \varphi = - |\xi'|_{\tilde{g}} \tilde{g} \dot{\sigma}_{y',-\xi'/|\xi'|_{\tilde{g}}} (t)$, which is exactly the opposite direction where the wave propagates. Since we assume $\nabla_{x'} \varphi $ points in the direction that $x^1, x^2$ increase, the wave propagates in the counterclockwise direction.
Therefore, we have a retrograde elliptical motion as before. 
\begin{remark}\label{solve_eikonal}
Notice in (\ref{angle_time}), the phase function $\varphi$ is always the dominated term for large $\xi'$.
By (\ref{eikonal}), the phase function $\varphi$ satisfies the Eikonal equation
\begin{align*}
\varphi_t   = |\nabla_{x'} \varphi|_{\tilde{g}}, \qquad \text{ with } \varphi(0, x, \xi) = x \cdot \xi.
\end{align*}
By \cite{GOnotes}, we solve it locally by the method of characteristics.
Let $H(t, x, \tau, \eta) = \tau - c_R|{x'}|_{\tilde{g}}$ be the Hamiltonian.
First we find the Hamiltonian curves by solving the system
\begin{align*}
&\dot{t}(s) = H_\tau = 1, &\dot{x'}(s) = H_{\xi'} = -g^{-1} \xi'/|\xi'|_{\tilde{g}},\\
&\dot{\tau}(s) = -H_t = 0, &\dot{\eta'}(s) = -H_{x'} = (\partial_{x'} g^jk) \xi_j \xi_k /|\xi'|_{\tilde{g}},
\end{align*}
where $s$ is the parameter 
and we set $\eta(s) = \nabla_{x'} \varphi(t, x', \xi')$.
This system corresponds to the unit speed geodesic flow
\[
x'(t) = \sigma_{y',-\xi'/|\xi'|_{\tilde{g}}} (t), \quad \eta'(t) = -|\xi'|_{\tilde{g}} \tilde{g} \dot{\sigma}_{y',-\xi'/|\xi'|_{\tilde{g}}} (t),
\]
where $y' = x'(0)$.
\end{remark}

\subsection{The inhomogeneous problem}\label{section_R_ih}
In this subsection, we solve the inhomogeneous problem (\ref{dn_eq}). 
We apply Proposition \ref{prop_solution_ih}
to the first equation in (\ref{hyperbolic_eq}) with zero initial condition. 
Recall $\tilde{l}_1$ is defined in (\ref{def_fl}) and $e_0$ in (\ref{def_e0}).
Then the first equation with with zero initial condition has a unique microlocal solution
\begin{align}\label{solution_firsteq}
h_1(t,x') = \int H(t-s) a(t-s, x', \xi') e^{i(\varphi(t-s, x', \xi') - y' \cdot \xi')} ((I - \tilde{e})^{-1}  e_0^{-1} \tilde{l}_1)(s,y')  \diff y' \diff \xi' \diff s,
\end{align}
where the phase function $\varphi(t,x',\xi')$ and the amplitude $a(t,x',\xi')$ are given by Proposition \ref{prop_solution_ih} with the hyperbolic operator being $\partial_t - i c_R(t,x')\sqrt{-\Delta_{x'}} + r(t, x', D_t, D_x')$.
We can also write the solution as $h_1 = L_{\varphi,a} (I - \tilde{e})^{-1} e_0^{-1} \tilde{l}_1$ by (\ref{def_L}). This proves the following theorem.

\begin{thm}\label{thm_ihm}
    Suppose $l(t,x') \in \mathcal{E}'((0,T) \times \mathbb{R}^2, \mathbb{C}^3)$ microlocally supported in the elliptic region.
    The inhomogeneous system (\ref{dn_eq}) with zero initial condition at $t=0$
    admits a unique microlocal solution
    \begin{align}\label{solution_f}
    \begin{pmatrix} 
    f_1\\
    f_2\\
    f_3
    \end{pmatrix} 
    =
    \widetilde{W}
    \begin{pmatrix} 
     L_{\varphi, a}(I - \tilde{e})^{-1} e_0^{-1} \tilde{l}_1 \\
    (\widetilde{m}_2 + r_2)^{-1} \tilde{l}_2\\
    (\widetilde{m}_3 + r_3)^{-1} \tilde{l}_3
    \end{pmatrix},
    \text{ where }
    \begin{pmatrix} 
    \tilde{l}_1\\
    \tilde{l}_2\\
    \tilde{l}_3
    \end{pmatrix}
    = \widetilde{W}^{-1}
    \begin{pmatrix} 
    l_1\\
    l_2\\
    l_3
    \end{pmatrix}
    \mod C^{\infty}.
    \end{align}
    
\end{thm}

Recall that we assume the microsupport of $l(t,x)$ is supported in $(0,T) \times \mathbb{R}^2$. 
Since $\widetilde{W}^{-1} = (\widetilde{W}^{-1})_{ij}$ for $i,j = 1,2,3$ is a matrix-valued $\Psi$DO, so does the microsupport of $\tilde{l}(t,x)$.
Similarly, the microsupport of $h_2, h_3, (I - \tilde{e})^{-1} e_0^{-1} \tilde{l}_1$ are supported in $(0,T) \times  \mathbb{R}^2$ as well.

In the following, we are going to find the polarization of the microlocal solution (\ref{solution_f}) outside the microsupport of $l(t,x')$. In other words, we only consider the solution when the source vanishes and we always assume $t\geq T$.
With this assumption, the Heaviside function in $k_L$ is negligible, i.e.
\begin{align*}
h_1(t,x) 
= F_{\varphi,a}(I - \tilde{e})^{-1} e_0^{-1} \tilde{l}_1,
\end{align*} 
where $F_{\varphi,a}$ is the FIO defined in (\ref{op_F}) associated with the canonical relation $C_F$.
Additionally, we have $
h_2 \in C^\infty, \ h_2 \in C^\infty
$ when $t \geq T$.
Then for $t \geq T$, the microlocal solution (\ref{solution_f}) is
\begin{align}\label{larger_t}
f = 
\begin{pmatrix} 
\widetilde{W}_{11}\\
\widetilde{W}_{21}\\
\widetilde{W}_{31}\\
\end{pmatrix}
F_{\varphi,a} (I - \tilde{e})^{-1} e_0^{-1} \tilde{l}_1
\mod C^\infty,
\end{align}
where $\tilde{l}_1 = (\widetilde{W}^{-1})_{11}l_1 + (\widetilde{W}^{-1})_{12}l_2 + (\widetilde{W}^{-1})_{13}l_3$.
To find out the leading term of the amplitude of the solution, we need to find the leading term $a_0$ of the amplitude of $h_1$, the solution to the following transport equation
\begin{align*}
(\partial_t  - \nabla_{\xi'} \lambda_1 \cdot \nabla_{x'}) a_0 - ( \sum_{|\alpha|=2} \frac{1}{\alpha!} \lambda_1^{(\alpha)} D_{x'}^\alpha \varphi - r_0(t,x',\lambda_1, \nabla_{x'} \varphi ))a_0 = 0,
\text{ with }a_0(0, x', \xi') = 1.
\end{align*}
Here the zero order term $r_0(t,x,\lambda_1, \xi)$ of the $\Psi$DO $r(t,x', D_t, D_{x'})$ is involved. The procesure of computing $r_0$ is in the Appendix \ref{symbolr}.
Then by Lemma \ref{fundl}, the leading term is given by 
\[
f \approx \int
\begin{pmatrix} 
\sigma_p(\widetilde{W}_{11})(t, x, \partial_t \varphi, \nabla_{x'} \varphi)\\
\sigma_p(\widetilde{W}_{21})(t, x, \partial_t \varphi, \nabla_{x'} \varphi)\\
\sigma_p
(\widetilde{W}_{31})(t, x, \partial_t \varphi, \nabla_{x'} \varphi)\\
\end{pmatrix}
a_0(t-s, x, \xi) e^{i\varphi}
\sigma_p(e_0^{-1})(t, x, \partial_t \varphi, \nabla_x \varphi)
\hat{\tilde{l}}_1(s, \xi') \diff s \diff \xi'.
\]
Recall $\sigma_{p}(\widetilde{W}) = \sigma_{p}({W})$ given in (\ref{W}) and $e_0$ in (\ref{e0}). The leading term equals to
\begin{align*}
&[\frac{a_0(t-s, x', \xi')\sigma_p(e_0^{-1})(t, x, \partial_t \varphi, \nabla_{x'} \varphi)}{\sqrt{(\beta \rho \tau^2 - m_1)^2 + |\xi'|_g^2\mu^2 \theta^2}}
\begin{pmatrix} 
i \mu \theta \xi_1/|\xi'|_g\\
i \mu \theta \xi_2/|\xi'|_g\\
\beta \rho \tau^2 
\end{pmatrix}
]\mid_{\tau =\partial_t \varphi,\ \xi' =  \nabla_{x'} \varphi } \\
=&
\frac{a_0(t-s, x', \xi')\sigma_p(e_0^{-1})(t, x, \partial_t \varphi, \nabla_{x'} \varphi)}{\sqrt{b(s_0)^2 \rho^2 c^2_R(t,x')+ \mu^2 \theta^2(c_R)}}
\begin{pmatrix} 
i \mu \theta(c_R) \nabla_{x_1} \varphi /{{|\nabla_{x'} \varphi|}_g}   \\
i \mu \theta(c_R) \nabla_{x_2} \varphi/{{|\nabla_{x'} \varphi|}_g}  \\
b(c_R)\rho c_R^2(t,x') 
\end{pmatrix}\\
=&
i \iota(c_R) a_0(t-s, x', \xi')
\begin{pmatrix} 
i\mu \theta(c_R) \nabla_{x_1} \varphi /{{|\nabla_{x'} \varphi|}_g} \\
i\mu \theta(c_R) \nabla_{x_2} \varphi /{{|\nabla_{x'} \varphi|}_g}  \\
 b(c_R)\rho c_R^2 
\end{pmatrix}
\coloneqq 
\mathcal{P}
\end{align*}
with 
\begin{align*}
\iota(c_R)  \equiv \frac{\sigma_p(e_0^{-1})(t, x', \partial_t \varphi, \nabla_x' \varphi)}{i \sqrt{b(c_R)^2 \rho^2 c^2_R(t,x')+ \mu^2 \theta^2(c_R)}}
&= \frac{1}{R'(c_R)}\sqrt{\frac{a(c_R) + b(c_R)}{b(c_R)}}
\end{align*}
where we use the notation $a(s), b(s), \theta(s), R(s)$ defined in (\ref{def_ab}) and combine (\ref{m1m2}), (\ref{Rayleigh_determinant}), (\ref{e0}).
Further, taking the phase function into consideration, we have the real and the imaginary part of the integrand equal to
\begin{align*}
&\Re(\mathcal{P})
=
-
\iota(c_R)  |a_0(t-s, x, \xi)|
\begin{pmatrix} 
\mu \theta(c_R)  \cos (\varphi + \upsilon)  \nabla_{x_1} \varphi /{{|\nabla_{x'} \varphi|}_g}   \\
\mu \theta(c_R)  \cos (\varphi + \upsilon)  \nabla_{x_2} \varphi /{{|\nabla_{x'} \varphi|}_g}  \\
b(c_R)\rho c_R^2  \sin (\varphi + \upsilon) 
\end{pmatrix},\\
&\Im(\mathcal{P})
= 
\iota(c_R)  |a_0(t-s, x, \xi)|
\begin{pmatrix} 
\mu \theta(c_R)  \cos (\varphi + \upsilon + \pi/2) \nabla_{x_1} \varphi /{{|\nabla_{x'} \varphi|}_g}    \\
\mu \theta(c_R)  \cos (\varphi + \upsilon + \pi/2)  \nabla_{x_2} \varphi /{{|\nabla_{x'} \varphi|}_g}  \\
b(c_R)\rho c_R^2  \sin (\varphi + \upsilon + \pi/2) 
\end{pmatrix},
\end{align*}
where we set $\varphi = \varphi(t-s, x, \xi) - y \cdot \xi$ and write $\upsilon = \arg(a_0)$. 
This implies a retrograde elliptical motion of $f$ as we stated before. 

Moreover, if we expand the term $\hat{\tilde{l}}_1$, then the microlocal solution in (\ref{larger_t}) can be written as 
\begin{align*}
f & = 
\begin{pmatrix} 
\widetilde{W}_{11}\\
\widetilde{W}_{21}\\
\widetilde{W}_{31}\\
\end{pmatrix}
F_{\varphi, a} (I - \tilde{e})^{-1} e_0^{-1} 
\begin{pmatrix} 
\widetilde{W}^{-1}_{11} &
\widetilde{W}^{-1}_{12} &
\widetilde{W}^{-1}_{13}
\end{pmatrix}
\begin{pmatrix} 
l_1\\
l_2\\
l_3\\
\end{pmatrix}
\mod C^\infty.
\end{align*}
Similarly, by the Fundamental Lemma we compute the leading term in the amplitude
\begin{align}
&a_0(t-s, x, \xi)
\begin{pmatrix} 
\sigma_p(\widetilde{W}_{11})\\
\sigma_p(\widetilde{W}_{21})\\
\sigma_p
(\widetilde{W}_{31})\\
\end{pmatrix}
\sigma_p(e_0^{-1})
\begin{pmatrix} 
\sigma_p(\widetilde{W}^{-1}_{11}) &
\sigma_p(\widetilde{W}^{-1}_{12}) &
\sigma_p(\widetilde{W}^{-1}_{13})
\end{pmatrix}
(t, x, \partial_t \varphi, \nabla_x \varphi) \nonumber \\
& = 
[a_0(t-s, x, \xi) \frac{\sigma_p(e_0^{-1})}{k_1^2}
\begin{pmatrix} 
i\mu \theta \xi_1/{{|\xi'|}_g} \\
i\mu \theta \xi_2/{{|\xi'|}_g} \\
\beta \rho \tau^2 - m_1\\
\end{pmatrix}
\begin{pmatrix} 
-i\mu \theta \xi_1 &
-i\mu \theta \xi_2 &
\beta \rho \tau^2 - m_1
\end{pmatrix}
]
\mid_{\tau =\partial_t \varphi,\ \xi' =  \nabla_{x'} \varphi }   \nonumber \\
& = 
 \frac{a_0(t-s, x, \xi)}{b(c_R) \rho c_R^2 R'(c_R)} 
\begin{pmatrix} 
i\mu \theta(c_R) \nabla_{x_1} \varphi /{{|\nabla_{x'} \varphi|}_g} \\
i\mu \theta(c_R) \nabla_{x_2} \varphi /{{|\nabla_{x'} \varphi|}_g} \\
b(c_R) \rho c_R^2\\
\end{pmatrix} \times \\
&
\qquad \qquad \qquad \qquad \qquad
\begin{pmatrix} 
\mu \theta(c_R)\nabla_{x_1} \varphi /{{|\nabla_{x'} \varphi|}_g} &
\mu \theta(c_R) \nabla_{x_2} \varphi /{{|\nabla_{x'} \varphi|}_g}  &
ib(c_R) \rho c_R^2
\end{pmatrix}, \nonumber 
\end{align}
where the last equality comes from (\ref{m1m2}), (\ref{Rayleigh_determinant}), (\ref{e0}), (\ref{def_k}).
Therefore, we have the following theorem.
\begin{thm}\label{thm_ihm_highorder}
    Assume everything in Theorem \ref{thm_ihm}. 
    For $t \geq T$, the displacement on the boundary equals to 
    \begin{align}
    f= \int
    & \frac{a_0(t-s, x, \xi)e^{i\varphi}}{b(c_R) \rho c_R^2 R'(c_R)} 
    \begin{pmatrix} 
    i\mu \theta(c_R) \nabla_{x_1} \varphi /{{|\nabla_{x'} \varphi|}_g}\\
    i\mu \theta(c_R) \nabla_{x_2} \varphi /{{|\nabla_{x'} \varphi|}_g} \\
    b(c_R) \rho c_R^2\\
    \end{pmatrix} 
    (
    \mu \theta(c_R) \hat{l}_1(s, \xi')\nabla_{x_1} \varphi /{{|\nabla_{x'} \varphi|}_g} +     \nonumber\\
    &
    \mu \theta(c_R) \hat{l}_2(s, \xi') \nabla_{x_1} \varphi /{{|\nabla_{x'} \varphi|}_g}  
    +    
    i b(c_R) \rho c_R^2 \hat{l}_3(s, \xi'))
    \diff s \diff \xi'+
    \text{ lower order terms }.
    \label{large_t_2}
    \end{align} 
\end{thm}
\begin{remark}
This theorem describes the microlocal polarization up to lower order terms of the displacement of Rayleigh waves on the boundary, when there is a source $(l_1, l_2, l_3)^T$ microlocally supported in the elliptic region with compact support.
Up to lower order terms, the displacement can also be regarded as a supposition of the real part and imaginary part of the leading term and each of them has a elliptical retrograde motion as we discussed before. 
Indeed, the leading term has a similar pattern of that in Theorem \ref{thm_cauchy}, except different scalar functions in each component.  
Additionally, one can see the compatible condition (\ref{compatible}) satisfied by the Cauchy data there actually means the Cauchy data can be regarded as produced by certain source, according to (\ref{larger_t}).
We mention that one can compute the polarization set defined in \cite{Dencker1982} and it corresponds to the direction of the polarization that we have here. 
\end{remark}
\subsection{The microlocal solution $u$}\label{u_construct}
With the boundary displacement $f$ given in Section \ref{section_R_cauchy} or \ref{section_R_ih}, 
in this subsection we are going to construct the microlocal solution $u$ to the elastic system (\ref{elastic_eq}) in the elliptic region.
This can be done by solving the boundary value problem, as is mentioned in Section \ref{section_pre_bvp}.

Recall in \cite{Stefanov2019} with $w = (w^s, w^p)^T = U^{-1} u$ we have the decoupled system
\begin{equation}\label{wave_w}
\begin{cases}
(\partial^2_t - c^2_s \Delta_g - A_s) w^s  = 0 \mod C^\infty,\\
(\partial^2_t - c^2_p \Delta_g - A_p) w^p  = 0 \mod C^\infty,
\end{cases}
\end{equation}
where $A_s \in \Psi^1$ is a $2\times 2$ matrix-valued $\Psi$DO and $A_s \in \Psi^1$ is a scalar one. 
The boundary value on $\{x^3=0 \}$ is given by $ w_b = U_{out}^{-1} f$.
For more details of $U$ and $U_{out}$ see \cite{Stefanov2019}. 
We construct the microlocal solution $w$ to (\ref{wave_w}) by the same procedure of constructing parametrices to the elliptic boundary value problems,
and then we get $u = Uw$. 
One can follow the construction in \cite[\S 3.2.3]{Stefanov2019} by solving the Eikonal equations with a complex phase function or use the following steps based on the construction in \cite[\S 7.12]{Taylor2010}.

First, in the semi-geodesic coordinates $(x', x^3)$, we have $g_{i3} = \delta_{i3}$ for $i = 1,2,3$, which enables us to regroup the above equations as
\begin{equation*}
\begin{cases}
(\partial^2_{x^3} - ({c^{-2}_s}\partial^2_t - \Delta_{g_{x'}}) - A_s) w^s  = 0 \mod C^\infty,\\
(\partial^2_{x^3} - ({c^{-2}_p} \partial^2_t - \Delta_{g_{x'}}) - A_p) w^p  = 0 \mod C^\infty,
\end{cases}
\end{equation*}
where $g_{x'}$ refers to $g$ restricted to $x'$ and we abuse the notations $A_s, A_p$ to denote the new lower order terms.
In the elliptic region, the operators $({c^{-2}_s}\partial^2_t - \Delta_{g_{x'}})$ and $({c^{-2}_p}\partial^2_t - \Delta_{g_{x'}})$ are self-adjoint operators with real positive symbols. 
Therefore they have a square root $B_s(t,x,D_{(t,x')}), B_p(t,x,D_{(t,x')}) \in \Psi^1$ respectively with 
\[
\sigma_p(B_s) = \sqrt{|\xi'|_g^2 - c_s^{-2}\tau^2} = \alpha, \quad \sigma_p(B_p) = \sqrt{|\xi'|_g^2 - c_p^{-2}\tau^2} = \beta.
\]
Further, we apply the elliptic operators $(\partial_{x^3} - B_s)^{-1}$ and $(\partial_{x^3} - B_p)^{-1}$ respectively to the equations above, 
to have
\begin{equation*}
\begin{cases}
(\partial_{x^3} + B_s + R_s) w^s  = 0 \mod C^\infty,\\
(\partial_{x^3} + B_p + R_p) w^p  = 0 \mod C^\infty,
\end{cases}
\end{equation*}
where $R_s, R_p$ are the new lower order terms of order zero.
Since $\partial_{x^3} + B_s $ is elliptic itself, one can write $(\partial_{x^3} + B_s + R_s) = (I - R_s(\partial_{x^3} + B_s)^{-1})(\partial_{x^3} + B_s)$. 
The same works for $\partial_{x^3} + B_p $.
This implies
\begin{equation*}
\begin{cases}
(\partial_{x^3} + B_s) w^s  = 0 \mod C^\infty,\\
(\partial_{x^3} + B_p) w^p  = 0 \mod C^\infty,
\end{cases}
\Rightarrow
\partial_{x^3} w = -Q(t,x,D_{t,x'})w \mod C^\infty,
\end{equation*}
where $Q$ is a matrix-valued $\Psi$DO with
\[
\sigma_p(Q) = 
-
\begin{pmatrix}
\alpha & 0 & 0\\
0 & \alpha & 0\\
0 & 0 & \beta 
\end{pmatrix}.
\]
By \cite[\S 7.12]{Taylor2010}, we look for 
\begin{align*}
w(t,x',x^3) = \int A(t,x',x^3,\tau, \xi') e^{i(t \tau + x' \cdot \xi')} \hat{w}_b(\tau, \xi') \diff \tau \diff \xi',
\end{align*}
where $A(t,x',x^3,\tau, \xi') \sim \sum_{j \geq 0} A_j(t,x',x^3,\tau, \xi')$ is constructed inductively.
More specifically, the leading term $A_0(t,x',x^3,\tau, \xi')$ is given by 
\[
(\partial_{x^3} - Q) A_0 = 0, \quad A_0(t,x',0,\tau, \xi')=I.
\]
For $j \geq 1$, we solve $A_j$ from 
\[
(\partial_{x^3} - Q) A_j = R_j, \quad A_0(t,x',0,\tau, \xi')=I,
\]
where $R_j$ is the reminder term related to $A_0, \ldots, A_{j-1}$.
It is shown in \cite{Taylor2010} that for $j \geq 0$ there exists $B_j(t,x',x^3,\tau, \xi')$ such that
\[
A_j = B_j e^{-C \sqrt{1 + |\tau|^2 + |\xi'|^2}}, \quad \text{with } (x^3)^kD_{x^3}^l B_j \text{ bounded in } S^{-j-k+l}_{1,0},
\]
where $C$ is the constant satisfying $\alpha > \beta > C |(\tau, \xi')|$. 
Such $C$ exists since there exists a conically compact neighborhood $K$ containing $\text{\( \WF \)}(f)$ in the elliptic region. 

When $f$ is given in (\ref{solution_f}), the parametrix $u = U w$ we construct above should satisfy
\begin{equation*}
\begin{cases}
u_{tt} - Eu  = 0  \mod C^\infty & \mbox{in } \mathbb{R}_t \times \Omega,\\
N u = l \mod C^\infty &\mbox{on } \Gamma,\\
u|_{t<0} = 0.
\end{cases}
\end{equation*}
In particular, it has the same microlocal behavior as the true solution to (\ref{elastic_eq}) by the justification of parametrices established in \cite{Taylor1979}.

\subsection{Flat case with constant coefficients}
In this subsection, we suppose the boundary $\Gamma$ is flat given by $x^3 = 0$ with Euclidean metric.
Suppose the parameters $\lambda, \mu, \rho$ are constants. 

In this case, by using the partition of unity, the elastic wave equation can be fully decoupled as 
\[
U_0^{-1} E U_0 = 
\begin{pmatrix}
-c_s^2 \Delta I_2 & 0\\
0 & -c_p^2 \Delta 
\end{pmatrix}
\]
and 
the solution is $u = U_0 w$, where
\begin{align}\label{flatw}
w(t,x) = 
\begin{pmatrix}
w_1\\
w_2\\
w_3\\
\end{pmatrix}
= \int e^{i(t \tau + x' \cdot \xi')}
\begin{pmatrix}
e^{-\widetilde{\alpha}{x^3}}   \hat{f}_1(\tau, \xi') \\
e^{-\widetilde{\alpha}{x^3}}   \hat{f}_2(\tau, \xi') \\
e^{-\widetilde{\beta}{x^3}}   \hat{f}_3(\tau, \xi') \\
\end{pmatrix}
\diff \tau \diff \xi'
\end{align}
with $f_i$ as the solution to $\Lambda f = l$, where
\begin{align*}
&\widetilde{\alpha} = \sqrt{|\xi'|^2 - c_{s}^{-2} \tau^2},
\quad \widetilde{\beta} = \sqrt{|\xi'|^2 - c_{p}^{-2} \tau^2}.
\end{align*}

To solve $\Lambda f = l$, we perform the same procedure as before. 
By verifying $\sigma_{p-1}(\Lambda) = 0$, we have
\[
W^{-1} \Lambda W = 
\begin{pmatrix}
e_0(\partial_t - ic_R\sqrt{-\Delta_{x'}}) & 0 & 0\\
0 & \widetilde{m}_2 & 0\\
0& 0 & \widetilde{m}_3
\end{pmatrix},
\]
where $e_0$ is given in (\ref{e0}) and we use the notations in (\ref{def_ab}) as before.
Combining the microlocal solution (\ref{goc1})  to the homogeneous hyperbolic equation and that of the inhomogeneous one with zero initial condition in (\ref{solution_firsteq}), we have the solution to the inhomogeneous one with arbitrary initial condition
\begin{equation*}
\begin{cases}
(\partial_t - i c_R \sqrt{- \Delta_{x'}} ) h_1 = e_0^{-1} (W^{-1} l)_1 \equiv g, \quad t >0 \\
h_1(0,x') = h_{1,0}(x')
\end{cases}
\end{equation*}
has the following solution
\begin{align}\label{flat_h1}
h_1(t,x') &= \int e^{i(t c_R {{|\xi'|}_g}  + x' \cdot \xi')} \hat{h}_{1,0}(\xi') \diff \xi' + 
\int H(t-s) e^{i((t-s)c_R {{|\xi'|}_g}  + x' \cdot \xi')} \hat{g}(s,\xi') \diff s \diff \xi'
\end{align}
and one can show
\[
\hat{h}_1(\tau, \xi') = \delta(c_R|\xi'| - \tau) \hat{h}_{1,0}(\xi') + \hat{H}(\tau - c_R{{|\xi'|}} ) \hat{g}(\tau, \xi')
\]
by directly taking the Fourier transform.
Since all $\Psi$DOs involved here have symbols free of $x'$, then they are Fourier multipliers and we have
\[
\hat{g}(\tau, \xi') = \sigma(e_0^{-1})\sum_k \sigma(W^{-1})_{1k} \hat{l}_k(\tau, \xi'),
\]
where $\sigma(W)$ given in (\ref{W}) is unitary and therefore $\sigma(W^{-1}) = \sigma(W^*) = \sigma(W)^*$.

The last two equations after we diagonalize the DN map have the following solutions
\[
\hat{h}_2(\tau, \xi')  = \sigma(\tilde{m}_2)^{-1} \sum_k \sigma(W^{-1})_{2k} \hat{l}_k(\tau, \xi'), 
\quad \hat{h}_3(\tau, \xi')  = \sigma(\tilde{m}_3)^{-1} \sum_k \sigma(W^{-1})_{3k} \hat{l}_k(\tau, \xi').
\]
Thus, the displacement on the boundary is given by $\hat{f}(\tau, \xi') = \sigma(W) \hat{h}(\tau, \xi')$.
\begin{example}
    In the following assume we have a time-periodic source
    \[
    l(t,x')
    = 
    \begin{pmatrix} 
    A_1, &
    A_2, &
    A_3
    \end{pmatrix}^T
    e^{ipt}\delta(x_1),
    \]
    where $p$ is a positive number and $A_1, A_2, A_3$ are constants.
    This gives us a line source periodic in $t$ on the boundary.
    Furthermore, we assume $A_1 = A_2 = 0$. In this case $\hat{{l}}_3 (s, \xi')= A_3e^{ips} \delta(\xi_2)$.
    Since the amplitude $a_0(t-s,x', \xi') \equiv 1$,
    by (\ref{large_t_2},\ref{flat_h1}) the displacement away from the support of the source up to lower order terms equals to 
    \begin{align*}
    f(t, x')
    & = \int
    \frac{ e^{i(t c_R {{|\xi'|}}  + x' \cdot \xi')}  }{\sqrt{b(c_R)(b(c_R) + a(c_R))} \rho c_R^2 } 
    \begin{pmatrix} 
    i\mu \theta(c_R) \xi_1/{{|\xi'|}} \\
    i\mu \theta(c_R) \xi_2/{{|\xi'|}} \\
    b(c_R) \rho c_R^2\\
    \end{pmatrix}
    \hat{h}_{1,0}(\xi') \diff \xi' + \text{lower order terms}\\
    & +  \int
    H(t-s)
    \frac{e^{i((t-s)c_R {{|\xi'|}}  + x' \cdot \xi')} }{ b(c_R) \rho c_R^2 R'(c_R)} 
    \begin{pmatrix} 
    i\mu \theta(c_R) \xi_1/{{|\xi'|}} \\
    i\mu \theta(c_R) \xi_2/{{|\xi'|}} \\
    b(c_R) \rho c_R^2\\
    \end{pmatrix}
    i b(c_R) \rho c_R^2
    A_3
    e^{ips}\delta(\xi_2)
    \diff s
    \diff \xi_1 \diff \xi_2.\\
    & = \int
    \frac{ e^{i(t c_R {{|\xi_1|}}  + x_1  \xi_1)}  }{\sqrt{b(c_R)(b(c_R) + a(c_R))} \rho c_R^2 } 
    \begin{pmatrix} 
    i\mu \theta(c_R) \xi_1/{{|\xi_1|}} \\
    0 \\
    b(c_R) \rho c_R^2\\
    \end{pmatrix}
    \hat{h}_{1,0}(\xi') \diff \xi' + \text{lower order terms}\\
        &+
        \frac{A_3 e^{ipt}}{R'(c_R)}  
        \begin{pmatrix} 
        \mu \theta(c_R) \big(e^{i p\frac{ x_1}{c_R}} - e^{-i p\frac{x_1}{c_R}}\big) \\
        0\\
        {  - i b(c_R) \rho c_R^2}\big( e^{i p\frac{x_1 }{c_R}} + e^{- i p\frac{x_1 }{c_R}}\big)  \\
        \end{pmatrix} 
        +
        \frac{A_3 e^{ipt}}{R'(c_R)}
        \int
        \frac{ e^{i x_1 \cdot \xi_1}}{p - c_R|\xi_1|} 
        \begin{pmatrix} 
        i\mu \theta(c_R) \xi_1/|\xi_1|\\
        0\\
        b(c_R) \rho c_R^2\\
        \end{pmatrix}
        \diff \xi_1.\\
    \end{align*}
    If we choose the Cauchy data as the inverse Fourier transform of the distribution \[
    \hat{h}_{1,0}(\xi') = -\frac{A_3 {e^{i(p - c_R|\xi_1|)t}}}{R'(c_R){(p - c_R |\xi_1|)}}\sqrt{b(c_R)(b(c_R) + a(c_R))}\rho c_R^2
    \] then we have
\begin{align*}
    f(t, x')
    & = 
    \frac{A_3 e^{ipt}}{R'(c_R)}  
    \begin{pmatrix} 
    \mu \theta(c_R) \big(e^{i p\frac{ x_1}{c_R}} - e^{-i p\frac{x_1}{c_R}}\big) \\
    0\\
    {  - i b(c_R) \rho c_R^2}\big( e^{i p\frac{x_1 }{c_R}} + e^{- i p\frac{x_1 }{c_R}}\big)  \\
    \end{pmatrix}  
    + \text{ lower order terms,}
    \end{align*}
    which coincides with (77) and (78) in \cite{1904a}.
\end{example}
\begin{example}
    In this example assume we have 
    $
    l
    = 
    \begin{pmatrix} 
    A_1\\
    A_2\\
    A_3
    \end{pmatrix}
    \delta(t)\delta(x_1),
    $
    where $A_1, A_2, A_3$ are constants.
    This gives us a source supported at $t=0$ on the boundary. In this case, since $\hat{{l}}_3 (s, \xi')= A_3\delta(s) \delta(\xi_2)$ and $A_1 = A_2 = 0$, by (\ref{large_t_2}) the displacement for $t>0$ equals to 
    \begin{align*}
    f
    & =  \int
    \frac{e^{i((t-s)c_R {{|\xi'|}}  + x' \cdot \xi')} }{b(c_R) \rho c_R^2 R'(c_R)} 
    \begin{pmatrix} 
    i\mu \theta(c_R) \xi_1/{{|\xi'|}} \\
    i\mu \theta(c_R) \xi_2/{{|\xi'|}} \\
     b(c_R) \rho c_R^2\\
    \end{pmatrix}
     ib(c_R) \rho c_R^2
    A_3
    \delta(s)\delta(\xi_2)
    \diff s
    \diff \xi_1 \diff \xi_2 \mod C^\infty\\
    & =  \int
    \frac{A_3 }{R'(c_R)} e^{i(tc_R |\xi_1| + x_1 \cdot \xi_1)} 
    \begin{pmatrix} 
    -\mu \theta(c_R) \sgn\xi_1\\
    0\\
    ib(c_R) \rho c_R^2\\
    \end{pmatrix}
    \diff \xi_1 
    =
    \begin{pmatrix} 
    {- A_3 \mu \theta(c_R) I_1}/{R'(c_R)} \\
    0\\
    {i A_3 b(c_R) \rho c_R^2  I_2}/{R'(c_R)}  \\
    \end{pmatrix} \mod C^\infty,
    \end{align*}
    where 
    \begin{align*}
    &I_1 = \int e^{i(tc_R |\xi_1| + x' \cdot \xi_1)}  \sgn\xi_1 \diff \xi_1 =
    \int e^{i( \xi_1(tc_R + x_1))}  H(\xi_1)\diff \xi_1  - \int e^{i( \xi_1(-tc_R + x_1))}  H(-\xi_1)\diff \xi_1\\
    &=  \pi(\delta(tc_R - x_1)  - \delta(tc_R + x_1))  + i(\text{p.v.} \frac{1}{tc_R + x_1} - \text{p.v.} \frac{1}{tc_R - x_1}),\\
    &I_2 = \int e^{i(tc_R |\xi_1| + x' \cdot \xi_1)} \diff \xi_1 =
    \int e^{i( \xi_1(tc_R + x_1))}  H(\xi_1)\diff \xi_1  + \int e^{i( \xi_1(-tc_R + x_1))}  H(-\xi_1)\diff \xi_1\\
    &= - \pi(\delta(tc_R - x_1) + \delta(tc_R + x_1)) + i(\text{p.v.} \frac{1}{tc_R - x_1} + \text{p.v.} \frac{1}{tc_R + x_1}).\\
    \end{align*}
\end{example}

\section{Stoneley waves}\label{section_stoneley}
In this section, we assume $\Gamma$ is an interface between two domains $\Omega_{-}, \Omega_{+}$. 
Locally, $\Gamma$ can be flatten as $x^3=0$ and $\Omega_+$ is the positive part. 
For the density and Lam\'{e} parameters, we have $\rho_+, \lambda_+, \mu_+$ in $\Omega_{+}$ and 
$\rho_-, \lambda_-, \mu_-$ in $\Omega_{-}$, which are functions smooth up to $\Gamma$. 
Let $u^{\pm}$ be $u$ restricted to $\Omega_\pm$. 

Suppose there are no incoming waves but boundary sources $l, q \in \mathcal{E'}((0,T) \times \mathbb{R}^2, \mathbb{C}^3)$ microlocally supported in the elliptic region, i.e. we are finding the outgoing microlocal solution $u^\pm$ for the elastic equation with transmission conditions 
\begin{equation}\label{elastic_trans}
\begin{cases}
\partial_t^2 u^{\pm} - Eu^\pm  = 0 & \mbox{in } \mathbb{R}_t \times \Omega_{\pm},\\
[u] = l, \quad [N u] = q &\mbox{on } \Gamma,\\
u|_{t<0} = 0,
\end{cases}
\end{equation}
where $[v]$ denote the jump of $v$ from the positive side to the negative side across the surface $\Gamma$.
By (9.2) in \cite{Stefanov2019}, with no incoming waves the transmission conditions can be written in the form of 
\begin{equation}\label{transmission}
\begin{pmatrix} 
U_{out}^{+} & -U_{out}^- \\
M_{out}^+  & - M_{out}^-
\end{pmatrix}
\begin{pmatrix} 
w_{out}^+ \\
w_{out}^-
\end{pmatrix}
=
\begin{pmatrix} 
l\\
q
\end{pmatrix}
\Longrightarrow
\begin{pmatrix} 
I & -I \\
\Lambda^+  & - \Lambda^-\\
\end{pmatrix}
\begin{pmatrix} 
f^+ \\
f^-
\end{pmatrix}
=
\begin{pmatrix} 
l\\
q
\end{pmatrix},
\end{equation}
if we set
\[
f^+= U_{out}^{+} w_{out}^+ = u^+|_{\Gamma}, \quad f^-= U_{out}^{-} w_{out}^- = u^-|_{\Gamma}.
\]
This implies that if we can solve $f^\pm$ from (\ref{transmission}), then the solution $u_\pm$ to (\ref{elastic_trans}) can be solved by constructing microlocal outgoing solutions to the boundary value problems with Dirichlet b.c. $f^+, f^-$ in $\Omega_{+}, \Omega_{-}$ respectively. 
Since $x^3$ has positive sign in $\Omega_{+}$ and negative sign in $\Omega_{-}$, to have evanescent modes in both domain, we choose $\xi_{3, \pm}^s, \xi_{3, \pm}^p$ with opposite signs
\begin{align*}
\xi_{3, \pm}^s = \pm i \alpha_\pm \equiv  \pm i\sqrt{|\xi'|_g^2 - c_{s, \pm}^{-2} \tau^2},  \quad \xi_{3, \pm}^p = \pm i \beta_\pm \equiv \pm i \sqrt{|\xi'|_g^2 - c_{p, \pm}^{-2} \tau^2},
\end{align*}
where 
\[
c_{s,\pm} = \sqrt{\mu_\pm/\rho_\pm}, \quad c_{p,\pm} = \sqrt{(\lambda_\pm + 2 \mu_\pm)/\rho_\pm}.
\]
Then the principal symbols $\sigma_p(\Lambda^\pm)$ are 
\[
\frac{i}{|\xi'|_g^2 - \alpha_\pm \beta_\pm}
\begin{pmatrix} 
\pm(\mu_\pm (\alpha_\pm - \beta_\pm)\xi_2^2 + \beta_\pm \rho_\pm \tau^2) &
\pm\mu_\pm\xi_1\xi_2(\beta_\pm - \alpha_\pm)& 
-i\mu_\pm \xi_1 \theta_\pm\\
\pm \mu_\pm\xi_1\xi_2(\beta_\pm - \alpha_\pm) &
\pm (\mu_\pm (\alpha_\pm - \beta_\pm)\xi_2^2 + \beta_\pm \rho_\pm \tau^2) 
& -i\mu_\pm \xi_2 \theta_\pm\\ 
i \mu_\pm \theta_\pm \xi_1 & i \mu_\pm \theta_\pm \xi_2 & \pm\alpha_\pm \rho_\pm \tau^2
\end{pmatrix}
,
\]
where 
\[
\theta_\pm = |\xi'|_g^2 + \alpha_\pm^2 - 2 \alpha_\pm \beta_\pm = 2 |\xi'|_g^2 -  c_{s,\pm}^{-2} \tau^2 - 2 \alpha_\pm \beta_\pm.
\]

To solve (\ref{transmission}), first we multiply the equation by an invertible matrix to have
\begin{align*}
&\begin{pmatrix} 
I & 0\\
- \Lambda^+  & I\\
\end{pmatrix}
\begin{pmatrix} 
I & -I \\
\Lambda^+  & - \Lambda^-\\
\end{pmatrix}
\begin{pmatrix} 
f^+ \\
f^-
\end{pmatrix}
=
\begin{pmatrix} 
l\\
q - \Lambda^+ l
\end{pmatrix}.\\
\Rightarrow &\begin{pmatrix} 
I &  - I \\
0  & \Lambda^+- \Lambda^-\\
\end{pmatrix}
\begin{pmatrix} 
f^+ \\
f^-
\end{pmatrix}
=
\begin{pmatrix} 
l\\
q - \Lambda^+ l
\end{pmatrix}
\Rightarrow
\begin{cases}
f^+ = l + f^-\\
(\Lambda^+- \Lambda^-)f^- = q - \Lambda^+ l
\end{cases}
\end{align*}
In the following, first we solve $f^-$ from 
\begin{align} \label{lambda_2}
(\Lambda^+- \Lambda^-)f^- = q - \Lambda^+ l
\end{align}
microlocally and then we have 
$f^+$. 
This gives the microlocal outgoing solution to (\ref{transmission}).
\subsection{Diagonalization of $\Lambda^+-\Lambda^-$}
Recall the calculation before, the principal symbol of the DN map can be partially diagonalized by the matrix $V_0$. 
By the same trick, we have
\begin{equation*}
V_0^* \sigma_{p}(\Lambda^+- \Lambda^-) V_0 = 
\begin{pmatrix} 
 {M} & 0\\
0 & (\mu_+ \alpha_+  + \mu_- \alpha_- )
\end{pmatrix}
\end{equation*}
where
\[
{M} = 
\frac{1}{|\xi'|_g^2 - \alpha_+ \beta_+}
\begin{pmatrix} 
\beta_+ \rho_+ \tau^2 & -i{{|\xi'|}_g} \mu_+ \theta_+ \\
i{{|\xi'|}_g} \mu_+ \theta_+ &  \alpha_+ \rho_+ \tau^2
\end{pmatrix}
-
\frac{1}{|\xi'|_g^2 - \alpha_- \beta_-}
\begin{pmatrix} 
- \beta_- \rho_- \tau^2 & -i{{|\xi'|}_g} \mu_- \theta_- \\
i{{|\xi'|}_g} \mu_-\theta_- & - \alpha_- \rho_- \tau^2
\end{pmatrix}.
\]
Let $s = \frac{\tau}{{{|\xi'|}_g} }$ as before.
We follow the similar argument as in \cite{Chadwick1994} to show the $2 \times 2$ matrix $M$ has two distinct eigenvalues and only one of them could be zero  for $0 < s <\min c_{s, \pm}$.
Define the following functions of $s$ related to $\alpha_\pm, \beta_\pm, \theta_\pm$
\begin{align}\label{def_ab_S}
a_\pm(s) = \sqrt{1 - c_{s,\pm}^{-2} s^2},
\quad b_\pm(s) = \sqrt{1 - c_{p,\pm}^{-2} s^2},
\quad \theta_\pm(s)
= 2  -  c_{s,\pm}^{-2} - 2 a_\pm (s) b_\pm (s).
\end{align}
Set
\begin{align*}
N_\pm (s) &= 
\frac{1}{1 - a_\pm b_\pm}
\begin{pmatrix} 
b_\pm(s) \rho_\pm s^2 & - i\mu_\pm \theta_\pm(s) \\
- i\mu_\pm \theta_\pm(s) &  a_\pm(s) \rho_\pm s^2
\end{pmatrix}\\
&=
\begin{pmatrix} 
\frac{b_\pm(s) \rho_\pm s^2 }{1 - a_\pm b_\pm}& - i(2 \mu_\pm - \frac{\rho_\pm s^2 }{1 - a_\pm b_\pm} ) \\
i(2 \mu_\pm - \frac{\rho_\pm s^2 }{1 - a_\pm b_\pm} ) &  \frac{a_\pm(s) \rho_\pm s^2 }{1 - a_\pm b_\pm}
\end{pmatrix}
,
\end{align*}
and it follows the matrix $M$ can be represented by 
\begin{align*}
{M} 
&= {{|\xi'|}_g} 
\begin{pmatrix} 
\frac{b_+(s) \rho_+ s^2 }{1 - a_+ b_+} + \frac{b_-(s) \rho_- s^2 }{1 - a_- b_-}& 
- i(2 (\mu_+ - \mu_-) - (\frac{\rho_+ s^2 }{1 - a_+ b_+} -\frac{\rho_- s^2 }{1 - a_- b_-} ) ) \\
i(2 (\mu_+ - \mu_-) - (\frac{\rho_+ s^2 }{1 - a_+ b_+} -\frac{\rho_- s^2 }{1 - a_- b_-} ) ) &  
\frac{a_+(s) \rho_+ s^2 }{1 - a_+ b_+} + \frac{a_-(s) \rho_- s^2 }{1 - a_- b_-}
\end{pmatrix}\\
&= {{|\xi'|}_g} (N_+(s) + N_-^T(s))  \equiv
\begin{pmatrix} 
M_{11}& M_{12}\\
M_{21} &  M_{22}
\end{pmatrix}, \numberthis \label{M}
\end{align*}
where we denote the entry of $M$ by $M_{ij}$ for $i,j = 1,2$ and the two eigenvalues by ${m}_1(s)$ and ${m}_2(s)$.
Notice 
\begin{align*}
\varrho &= \sqrt{(M_{11} - M_{22})^2 + 4M_{12}M_{21}} \\
& = (\frac{(b_+(s) -a_+(s)) \rho_+ s^2 }{1 - a_+ b_+}  -  \frac{(b_-(s) -a_-(s)) \rho_- s^2 }{1 - a_- b_-})^2 + (2 (\mu_+ - \mu_-) - (\frac{\rho_+ s^2 }{1 - a_+ b_+} -\frac{\rho_- s^2 }{1 - a_- b_-} ) )^2
\end{align*}
is always nonnegative.
The eigenvalues of $M$ can be written in the following specific form
\begin{align}\label{def_m1_s}
&m_1(t,x', \tau, \xi') = \frac{M_{11} + M_{22} - \sqrt{\varrho}}{2} \equiv {{|\xi'|}_g}  m_1(s), \\ 
&m_2(t,x', \tau, \xi') = \frac{M_{11} + M_{22} + \sqrt{\varrho} }{2} \equiv {{|\xi'|}_g}  m_2(s), \nonumber
\end{align}
where only $m_1(s)$ could be zero. 
More precisely, we have   $m_1(s)$ vanishes if and only if 
\begin{align}\label{detM_S}
0 = \det M = (1-a_+b_+)(1-a_-b_-) S(s)
\end{align}
is satisfied for some $s_0$, where
\begin{align}\label{eq_S}
S(s) &= 
((\rho_+ a_- + \rho_- a_+)(\rho_+ b_- + \rho_- b_+) - (\rho_+ - \rho_-)^2) s^4 \nonumber \\
&- 4(\mu_+ - \mu_-)^2(1-a_+b_+)(1-a_-b_-) \nonumber \\
&+ 4(\mu_+ - \mu_-)(\rho_+ (1-a_-b_-) -\rho_-(1-a_+b_+)) s^2.
\end{align}
Notice the factor $(1-a_+b_+)(1-a_-b_-)$ is always positive and therefore it is equivalent to $S(s_0) = 0$.
If such $s_0$ exists, it corresponds to the propagation speed $c_{ST}$ of the so called Stoneley waves, first proposed in \cite{1924a}. 
We call $ c_{ST}$ the \textbf{Stoneley speed}
and it is a simple zero by Proposition \ref{simplezero} in the following. 
This proposition of uniqueness of Stoneley waves is proved by the definiteness of $N_\pm(s)$ and first appears in \cite{1985a} and then is used in \cite{Chadwick1994}. 
Here we rewrite it in our notations.
\begin{pp}\label{simplezero}
    For $0 < s < \min c_{s, \pm}$, the eigenvalues $m_1(s), m_2(s)$ decrease as $s$ increases. 
    Only $m_1(s)$ can be zero.
    This happens when there is some $s_0$ such that 
    
    Particularly, if such $s_0$ exists, it is unique and is a simple zero of $m_1(s)$.
\end{pp}
\begin{proof}
    We claim the matrix $N_\pm(s)$ and their transposes satisfy 
    \begin{itemize}
        \item[(a)] the limit $N_\pm(0) \equiv \lim_{s \rightarrow 0 }N_\pm(s)$ exist and are positive definite,
        \item[(b)] for $0 < s < \min c_{s, \pm}$, the derivative $N'_\pm(s)$ is negative definite,
        \item[(c)] for $0 < s < \min c_{s, \pm}$, the trace $\Tr(N_\pm(s))$ is always positive.
    \end{itemize}  
    Then ${M}$ satisfy these conditions as well, which indicates its eigenvalues decrease as $s$ increases but their sum is always positive, i.e. at most one of them could be zero. The monotonic decreasing of eigenvalues implies the zero should be a simple one.
    
    Now we prove the claim. 
    For (a), we compute
    \begin{align*}
    \lim_{s \rightarrow 0 }\frac{\rho_\pm s^2 }{1 - a_\pm b_\pm} 
    &=\lim_{s \rightarrow 0 } \frac{2 \rho_\pm s}{-a_\pm'b_\pm - a_\pm b_\pm'} 
    = \lim_{s \rightarrow 0 } \frac{2 \rho_\pm }{-a_\pm''b_\pm -2 a_\pm'b_\pm'- a_\pm b_\pm''} \\
    &=\frac{2\mu_\pm(2\mu_\pm+\lambda_\pm)}{3\mu_\pm+\lambda_\pm}
    \equiv c_\pm.
    \end{align*}
    Assuming $\mu_\pm, \lambda_\pm >0$, we have $ \frac{4}{3} \mu_\pm <c_\pm < 2\mu_\pm$. 
    Then $\Tr(N_\pm(0)) = 2 c_\pm >0$ and $\det(N_\pm(0)) = 4 \mu_\pm (c_\pm - \mu_\pm) > 0 $.
    
    To prove (b), for convenience we change the variable $\iota = s^2$ 
    and $\widetilde{N}_\pm(\iota) = N_\pm(\sqrt{\iota})$ with $s >0$.
    Then it is sufficient to show $\widetilde{N}_\pm'(\iota)>0$. 
    Indeed we have
    \[
    \widetilde{N}_\pm(\iota)
    =
    \begin{pmatrix} 
    \tilde{b}_\pm(\iota) \kappa_\pm(\iota) & - i(2 \mu_\pm - \kappa_\pm(\iota) ) \\
    i(2 \mu_\pm - \kappa_\pm(\iota) ) &  \tilde{a}_\pm (\iota)\kappa_\pm(\iota) 
    \end{pmatrix},
    \]
    where we set 
    \[
    \tilde{a}_\pm(\iota)= \sqrt{1 - c_{s,\pm}^{-2} \iota},
    \quad \tilde{b}_\pm(\iota) = \sqrt{1 - c_{p,\pm}^{-2} \iota}, \quad \kappa_\pm(\iota) = \frac{\rho_\pm \iota }{1 - \tilde{a}_\pm \tilde{b}_\pm} .
    \]
    Then
    \[
    \widetilde{N}'_\pm(\iota)
    =
    \begin{pmatrix} 
    \tilde{b}'_\pm(\iota) \kappa_\pm(\iota)  + \tilde{b}_\pm(\iota) \kappa_\pm'(\iota) &  i 2 \kappa_\pm'(\iota) ) \\
    -i 2 \kappa_\pm'(\iota) ) &  \tilde{a}'_\pm (\iota)\kappa_\pm(\iota) + \tilde{a}_\pm (\iota)\kappa_\pm'(\iota) 
    \end{pmatrix}
    \]
    and
    \begin{align*}
    \kappa_\pm' &= \frac{\rho_\pm}{(1 - \tilde{a}_\pm \tilde{b}_\pm)}(1 - \frac{\tilde{a} _\pm}{2\tilde{b} _\pm}- \frac{\tilde{b} _\pm}{2\tilde{a}_\pm}),\\
    \tilde{a}'_\pm(\iota) k(\iota )&=\frac{-\rho_\pm c_{s,\pm}^{-2} \iota}{2\tilde{a}_\pm(1 - \tilde{a}_\pm \tilde{b}_\pm)}
    = \frac{\rho_\pm}{2(1 - \tilde{a}_\pm \tilde{b}_\pm)}(\tilde{a} _\pm- \frac{1}{\tilde{a}_\pm}).
    \end{align*}
    Therefore, the determinant and the transpose of $\widetilde{N}_\pm'(\iota)$ are
    \begin{align*}
    \det(\widetilde{N}_\pm'(\iota)) &= (\kappa_\pm')^2 \tilde{a}_\pm \tilde{b}_\pm + \kappa_\pm \kappa_\pm'\tilde{a}'_\pm \tilde{b}_\pm + \kappa_\pm \kappa_\pm'\tilde{a}_\pm \tilde{b}'_\pm - (\kappa_\pm')^2\\
    &=\frac{\rho_\pm^2}{(1 - \tilde{a}_\pm \tilde{b}_\pm)^2} \frac{1}{2 \tilde{a}_\pm \tilde{b}_\pm }(\tilde{a}_\pm - \tilde{b}_\pm )^2 > 0,\\
    \Tr(\widetilde{N}_\pm'(\iota)) &= \kappa_\pm'(\tilde{a}_\pm + \tilde{b}_\pm) + \kappa_\pm(\tilde{a}'_\pm + \tilde{b}_\pm') \\
    &= -\frac{\rho_\pm(\tilde{a}_\pm + \tilde{b}_\pm)}{2\tilde{a}_\pm  \tilde{b}_\pm(1 - \tilde{a}_\pm \tilde{b}_\pm)^2}((\tilde{a}_\pm - \tilde{b}_\pm)^2 + (\tilde{a}_\pm  \tilde{b}_\pm + 1)^2) <0,
    \end{align*}
    which indicates $\widetilde{N}_\pm'(\iota)$ is negative definite. 
    For (c), obviously we have
    \[
    \Tr(N_\pm)(s) = \frac{(a_\pm(s)+ b_\pm(s)) \rho_\pm s^2 }{1 - a_\pm b_\pm} >0 .
    \]
\end{proof}

If $S(s)  \neq 0$ for all  $0 < s < \min c_{s, \pm}$, 
then $\Lambda+ - \Lambda_-$ is an elliptic $\Psi$DO and the microlocal solution to (\ref{lambda_2}) is $f^- = (\Lambda_+ - \Lambda_-)^{-1} (l_2 - \Lambda_+ l_1)$. The singularities does not propagate.
Otherwise, if there exists $0 < c_{ST}  < \min c_{s, \pm}$ such that ${m}_1(c_{ST}) = 0$, then (\ref{transmission}) has a nontrivial microlocal solution that propagates singularities, analogously to the case of Rayleigh waves.

In the following suppose there exists a $c_{ST}$ such that $S(c_{ST}) = 0$. 
By Proposition \ref{simplezero}, this zero is simple so by the implicit function theorem it is a smooth function $s_{ST}(t,x')$ in a small neighborhood of a fixed point. 
This time the eigenvalues of the system is simply $\widetilde{m}_j = m_j(t,x',\tau,\xi')$ for $j=1,2,3$. 
Then we can write $\widetilde{m}_1(t,x',\tau,\xi')$ as a product similar to what we have before
\begin{align}\label{def_e0_s}
\widetilde{m}_1(t,x',\tau,\xi') =   e_0(t,x',\tau,\xi') i (\tau -c_{ST}(t,x') {{|\xi'|}_g} ),
\end{align}
where 
\begin{align}\label{e0_S}
e_0(t,x',\tau, \xi') = \frac{m_1(t,x',\tau, \xi')}{i (\tau - c_{ST}(t,x')|\xi'|_g)}  =\frac{m_1(s)}{i (s - c_{ST}(t,x'))}
\end{align}
To decouple the system as what we did in Section \ref{Rayleigh}, we need the following claim. 
Notice even without this claim, the analysis still holds since with only ${m}_1$ vanishing the last two eigenvalues always give us an elliptic $2$ by $2$ system.
\begin{claim}
    The three eigenvalues $\widetilde{m}_1(t,x',\tau, \xi'), \widetilde{m}_1(t,x',\tau, \xi'), \widetilde{m}_1(t,x',\tau, \xi')$
    of the matrix $\sigma_p(\Lambda_+ - \Lambda_-)$  are distinct near $s = c_{ST}$.
\end{claim}
\begin{proof}
    Obviously near $s_0$ we have $\widetilde{m}_1 \neq \widetilde{m}_2$ and $\widetilde{m}_1 \neq \widetilde{m}_3$.
    The values of $\widetilde{m}_2, \widetilde{m}_3$ may coincide but near $s_0$ they are separate by the following calculation
    \begin{align*}
    &\widetilde{m}_1(t,x',\tau, \xi') + \widetilde{m}_2(t,x',\tau, \xi') -\widetilde{m}_3(t,x',\tau, \xi') 
    = \Tr(M) - i (\mu_+ \alpha_+  + \mu_- \alpha_- )\\
    & =  {{|\xi'|}_g} \sum_{\nu = \pm} \frac{ (a_\nu(s)+ b_\nu(s))\rho_\nu s^2 - \mu_\nu a_\nu (1 - a_\nu b_\nu)}{1 - a_\nu b_\nu} \\
    & =  {{|\xi'|}_g} \sum_{\nu = \pm} \frac{ (a_\nu(s)+ b_\nu(s))\rho_\nu s^2 - \mu_\nu a_\nu(s)  +  \mu_\nu b_\nu(s)(1 - \frac{\rho_\nu}{\mu_\nu} s^2)}{1 - a_\nu(s) b_\nu(s)} \\
    & =  {{|\xi'|}_g} \sum_{\nu = \pm} \frac{ a_\nu(s)\rho_\nu s^2 + \mu_\nu (b_\nu(s) - a_\nu(s))  }{1 - a_\nu b_\nu} >0, \\
    \end{align*}
    where the last inequality holds, since by (\ref{def_ab_S}) we have
    $0 <a_\nu < b_\nu<1$ and $\widetilde{m}_1 = 0$ at $s = c_{ST}$.
\end{proof}
More specifically, this time we define
\[
V_1(t,x',\tau, \xi') = 
\begin{pmatrix} 
{ M_{21}}/{k_1} 
& {M_{21}}/{k_2}   & 0\\
(M_{11} -  \widetilde{m}_1)/{k_1} 
& (M_{11} -  \widetilde{m}_2)/{k_2}  & 0\\
0 & 0 & 1
\end{pmatrix},
\] 
where $M_{ij}$ is the entry of $M$ in (\ref{M}), for $i,j = 1,2$ and we define
\begin{align}\label{def_k_s}
k_j = \sqrt{|M_{11} - \widetilde{m}_j|^2 + |M_{21}|^2}, \quad k_j(s) = k_j/|\xi'|_g^2.
\end{align}
Then 
\begin{align}\label{W_s}
W(t,x',\tau, \xi') &=  V_0(t,x',\tau, \xi')V_1(t,x',\tau, \xi') \nonumber \\
&= 
\begin{pmatrix} 
{M_{21}\xi_1}/{({{|\xi'|}_g} k_1)} 
& {-M_{12}\xi_1}/{({{|\xi'|}_g} k_2)}   & {-\xi_2}/{{{|\xi'|}_g}}\\
{M_{21}\xi_2}/{({{|\xi'|}_g} k_1)}
& {-M_{12}\xi_2}/{({{|\xi'|}_g} k_2)}   & {\xi_1}/{{{|\xi'|}_g}}\\
(M_{11} -  \widetilde{m}_1)/{k_1} 
& (M_{11} -  \widetilde{m}_2)/{k_2}   & 0\\
\end{pmatrix}.
\end{align}

Let the operators $e_0(t,x', D_t, D_{x'}) \in \Psi^0$  with symbol 
$e_0(t,x', \tau, \xi')$ in (\ref{def_e0_s})
and $\widetilde{m}_j(t,x', D_t, D_{x'})\in \Psi^1$ with symbols $\widetilde{m}_j(t,x', \tau, \xi')$, for $j= 2,3$. 
By \cite{Taylor2017}, there exists an elliptic $\Psi$DO of order zero $\widetilde{W}(t,x', D_t, D_{x'})$ 
with the principal symbol equal to $W(t,x',\tau, \xi')$, such that
near some fixed $(t,x', \tau, \xi') \in \Sigma_R$, the operator $\Lambda^+ - \Lambda^-$ can be fully decoupled as 
\begin{align*}
\widetilde{W}^{-1}(\Lambda^+ - \Lambda^-) \widetilde{W}  =
\begin{pmatrix} 
e_0(\partial_t - i c_{ST}(t,x')\sqrt{-\Delta_{x'}}) + r_1 & 0 & 0\\
0 & \widetilde{m}_2 + r_2 & 0\\
0 & 0 & \widetilde{m}_3 + r_3
\end{pmatrix} 
\mod \Psi^{-\infty},
\end{align*}
where 
$r_1(t,x', D_t, D_{x'}), r_2(t,x', D_t, D_{x'}), r_3(t,x', D_t, D_{x'}) \in \Psi^0$
are the lower order term.
If we define 
\begin{align}\label{def_lambda_s}
r(t, x', D_t, D_x')= e_0^{-1} r_1 \in \Psi^0,
\end{align} 
in what follows,
then the first entry in the first row can be written as 
\begin{align}\label{hyperbolic_eq_S}
e_0(\partial_t - i c_{ST}(t,x')\sqrt{-\Delta_{x'}} + r(t, x', D_t, D_x')).
\end{align} 

\subsection{The Cauchy problem and the polarization}\label{section_S_cauchy}
In this subsection, similar to Subsection \ref{section_R_cauchy}, 
we first assume that the source exists for a limited time for $t<0$ and we have the Cauchy data $f^-|_{t=0}$ at $t=0$, i.e.
\begin{equation}\label{dn_cauchy_S}
(\Lambda_+ - \Lambda_-) f^- = 0, \text{ for } t > 0, \quad 
f^-|_{t=0} \text{ given }.
\end{equation}
Recall the diagonalization of $\Lambda_+ - \Lambda_-$ in before. 
Let 
\begin{align*}
h^- =
\begin{pmatrix}
h^-_1\\
h^-_2\\
h^-_3
\end{pmatrix} 
=
\widetilde{W}^{-1}
\begin{pmatrix}
f^-_1\\
f^-_2\\
f^-_3
\end{pmatrix} 
=  \widetilde{W}^{-1} f^-.
\end{align*}
The homogeneous equation $(\Lambda_+ - \Lambda_-) f^- = 0$ implies 
\begin{align}\label{eg_fh_S}
f^- = \widetilde{W} h^-
= \begin{pmatrix}
\widetilde{W}_{11}\\
\widetilde{W}_{21}\\
\widetilde{W}_{31}\\
\end{pmatrix} h^-_1 \mod C^\infty
\Leftarrow h^-_2 = h^-_3 =0 \mod C^\infty,
\end{align}
where $h^-_1$ solves the homogeneous first-order hyperbolic equation in (\ref{hyperbolic_eq_S}).
Notice in this case the hyperbolic operator is $\partial_t - i c_{ST}(t,x'){{|\xi'|}_g}  + r(t, x', D_t, D_{x'})$ with $r$ given by (\ref{def_lambda_s}).
If we have the initial condition $h^-_{1,0}\equiv h^-_1|_{t=0}$, then by the construction in Subsection \ref{subsec_he}, then
\begin{align}\label{eg_h1_S}
h^-_1(t,x') = \int a(t,x',\xi') e^{i\varphi(t,x',\xi')}  \hat{h}^-_{1,0}(\xi') \diff \xi' \mod C^\infty,
\end{align}
where the phase function $\varphi$ solves the eikonal equation (\ref{eikonal});
the amplitude $a  = a_0 + \sum_{j \leq -1}a_j$ solves the transport equation (\ref{transport_h_a0}) and (\ref{transport_h_aj}) with
$
\gamma  = ( \sum_{|\alpha|=2} \big(\frac{1}{\alpha!} c_R(t,x')D_{\xi'}^\alpha{{|\xi'|}_g}  D_{x'}^\alpha \varphi-r_0(t,x',\lambda_1, \nabla_{x'}\varphi)\big).
$

By the same analysis before, the initial condition of $f^-$ is related to that of $h^-$ by
\begin{align}\label{compatible_S}
f^-|_{t=0} = 
\begin{pmatrix}
\widetilde{W}_{11,0}\\
\widetilde{W}_{21,0}\\
\widetilde{W}_{31,0}\\
\end{pmatrix} 
h^-_{1,0} \mod C^\infty,
\end{align}
where the principal symbols are 
\begin{align*}
&\sigma_p(\widetilde{W}_{11,0}) = \sigma_p({W}_{11}) (0, x', c_R{{|\xi'|}_g} , \xi'), \qquad
\sigma_p(\widetilde{W}_{21,0}) = \sigma_p({W}_{21})(0, x', c_R{{|\xi'|}_g} , \xi'),\\
&\sigma_p(\widetilde{W}_{31,0})= \sigma_p({W}_{31})(0, x', c_R{{|\xi'|}_g} , \xi'). 
\end{align*}
We have the following theorem as an analog to Theorem \ref{thm_cauchy}.
\begin{thm}\label{thm_cauchy_S}
    Suppose $f^-(0,x')$ satisfies \text{(\ref{compatible_S})} with some $h^-_{1,0} \in \mathcal{E}'$. Then microlocally the homogeneous problem with Cauchy data \text{(\ref{dn_cauchy_S})} admits a unique microlocal solution
    \begin{align}\label{solution_cauchy_S}
    f^- &= \begin{pmatrix}
    \widetilde{W}_{11}\\
    \widetilde{W}_{21}\\
    \widetilde{W}_{31}\\
    \end{pmatrix}
    \int a(t,x',\xi') e^{i\varphi(t,x',\xi')}  \hat{h}^-_{1,0}(\xi') \diff \xi' \mod C^\infty \nonumber\\
    &=  \int
    \underbrace{ 
        \begin{pmatrix} 
        i \zeta_1 \nabla_{x_1} \varphi /{{|\nabla_{x'} \varphi |}_g}  \\
        i \zeta_1 
        \nabla_{x_2} \varphi /{{|\nabla_{x'} \varphi |}_g} \\
        \zeta_2
        \end{pmatrix} 
        \frac{a_0(t,x',\xi') e^{i\varphi(t,x',\xi')}}{k_1(c_{ST})}
    }_{\mathcal{P}}
    \hat{h}^-_{1,0}(\xi') \diff \xi'
    + \text{ lower order terms},
    \end{align}
    where we define
    \begin{align*}
    &\zeta_1 = (2 (\mu_+ - \mu_-) - (\frac{\rho_+ c_{ST}^2 }{1 - a_+(c_{ST}) b_+(c_{ST})} -\frac{\rho_- c_{ST}^2 }{1 - a_-(c_{ST}) b_-(c_{ST})} ) ),\\
    &\zeta_2 = (\frac{b_+(c_{ST}) \rho_+ c_{ST}^2 }{1 - a_+(c_{ST}) b_+v} + \frac{b_-(c_{ST}) \rho_- c_{ST}^2 }{1 - a_- (c_{ST})b_-(c_{ST})})
    \end{align*}
    as smooth functions w.r.t. $t, x', \xi'$ with $c_{ST}$ be the Stoneley speed, $a_+(s), b_+(s)$ defined in (\ref{def_ab_S}), $k_1(s)$ defined in (\ref{def_k_s}).
\end{thm}
This theorem describes the microlocal polarization of the displacement of a Stoneley wave on the intersurface.
Up to lower order terms, the displacement $f^-$ can also be regarded as a supposition of  $\Re(\mathcal{P})$ and $\Im(\mathcal{P})$, each of which has a elliptical retrograde motion as we discussed before. 
Indeed, the leading term of $f$ has a similar pattern of that of Rayleigh waves in Theorem (\ref{thm_cauchy}), except different scalar functions in each component. 


\subsection{The inhomogeneous problem}
In this subsection, first we are going to microlocally solve the inhomogeneous problem
\begin{align}\label{dn_s}
(\Lambda_+ - \Lambda_-) f^- = q - \Lambda^+ l, \text{ for } t>0, \quad f^-|_{t=0} = 0,
\end{align}
and the then the microlocal solution to (\ref{transmission}) and (\ref{elastic_trans}) can be derived as we stated before.
We perform the same analysis for the Rayleigh wave in the previous section.

Let 
\begin{align}\label{def_fl_s}
&h^- =
\begin{pmatrix}
h^-_1\\
h^-_2\\
h^-_3
\end{pmatrix} 
=
\widetilde{W}^{-1}
\begin{pmatrix}
f^-_1\\
f^-_2\\
f^-_3
\end{pmatrix} 
=  \widetilde{W}^{-1} f^-,
&\tilde{l} 
=
\begin{pmatrix}
\tilde{l}_1\\
\tilde{l}_2\\
\tilde{l}_3
\end{pmatrix} 
= \widetilde{W}^{-1}(l_2 - \Lambda^+ l_1),
\end{align} 
where $u_j$ is the component of any vector valued distribution $u$ for $j = 1, 2, 3$.
Solving $\Lambda f = l \mod C^\infty$ is microlocally equivalent to solving the following system
\begin{equation}\label{hyperbolic_eq_s}
\begin{cases}
(\partial_t - i c_{ST}(t,x')\sqrt{-\Delta_{x'}} + r(t, x', D_t, D_x')) h_1^- = e_0^{-1}\tilde{l}_1, &\mod C^\infty,\\
(\widetilde{m}_2 + r_2)h_2^- = \tilde{l}_2,  &\mod C^\infty,\\
(\widetilde{m}_3 + r_3)h_3^- = \tilde{l}_3, &\mod C^\infty.
\end{cases}
\end{equation}
In the last two equations, the operators $\widetilde{m}_j + r_j$ are elliptic so we have $h_j^- = (\widetilde{m}_j + r_j)^{-1} \tilde{l}_j, \mod C^\infty$ for $j=2,3$.
The first equation is a first-order inhomogeneous hyperbolic equation with lower order term, which can be solved by Duhamel's principle.
We apply Proposition \ref{prop_solution_ih} to have
\begin{align}\label{solution_firsteq}
h^-_1(t,x') & = \int H(t-s) a(t-s, x', \xi') e^{i(\varphi(t-s, x', \xi') - y' \cdot \xi')} ((I - \tilde{e})^{-1} e_0^{-1} \tilde{l}_1)(s,y')  \diff y' \diff \xi' \diff s \nonumber \\
& =  L_{\varphi, a}(I - \tilde{e})^{-1} e_0^{-1} \tilde{l}_1,
\end{align}
where the phase function $\varphi(t,x',\xi')$ and the amplitude $a(t,x',\xi')$ are given by Proposition \ref{prop_solution_ih} with the hyperbolic operator being $\partial_t - i c_{ST}(t,x')\sqrt{-\Delta_{x'}} + r(t, x', D_t, D_x')$; and $\tilde{e}$ is defined in (\ref{def_fl}) for the new hyperbolic operator.
We also write the solution as $h_1^- = L_{\varphi,a} (I - \tilde{e})^{-1} e_0^{-1} \tilde{l}_1$ by (\ref{def_L}). This proves the following theorem, as an analog to Theorem \ref{thm_ihm}, \ref{thm_ihm_highorder}.

\begin{thm}\label{thm_ihm_S}
    Suppose $l(t,x'), q(t,x') \in \mathcal{E}'((0,T) \times \mathbb{R}^2, \mathbb{C}^3)$ microlocally supported in the elliptic region.
    The inhomogeneous system (\ref{dn_s}) with zero initial condition at $t=0$
    admits a unique microlocal solution
    \begin{align}\label{solution_f_s}
    \begin{pmatrix} 
    f_1^-\\
    f_2^-\\
    f_3^-
    \end{pmatrix} 
    =
    \widetilde{W}
    \begin{pmatrix} 
    L_{\varphi, a}(I - \tilde{e})^{-1} e_0^{-1} \tilde{l}_1 \\
    (\widetilde{m}_2 + r_2)^{-1} \tilde{l}_2\\
    (\widetilde{m}_3 + r_3)^{-1} \tilde{l}_3
    \end{pmatrix},
    \text{ with }
    \begin{pmatrix} 
    \tilde{l}_1\\
    \tilde{l}_2\\
    \tilde{l}_3
    \end{pmatrix}
    =
    \widetilde{W}^{-1}(q - \Lambda^+ l)
    \mod C^{\infty},
    \end{align}
    where $\widetilde{W}$ has the principal symbol in (\ref{W_s}), $e_0$ defined in (\ref{def_e0_s}),
    and $\tilde{e}$ is constructed as in (\ref{etilde}).
    Then the microlocal solution to the transmission problem (\ref{elastic_trans}) can be constructed 
    as evanescent modes from the boundary value $f^-$ and $f^+= l + f^-$.
    In particular, for $t \geq T$, the displacement on the boundary has the leading term
    \begin{align*}
    f^- &= 
    \begin{pmatrix} 
    \widetilde{W}_{11}\\
    \widetilde{W}_{21}\\
    \widetilde{W}_{31}\\
    \end{pmatrix}
    F_{\varphi,a} (I - \tilde{e})^{-1} e_0^{-1} \tilde{l}_1 \\
    & = 
    \int 
    \begin{pmatrix} 
    i \zeta_1 \nabla_{x_1} \varphi /{{|\nabla_{x'} \varphi |}_g} \\
    i \zeta_1 \nabla_{x_2} \varphi /{{|\nabla_{x'} \varphi |}_g} \\
    \zeta_2\\
    \end{pmatrix}
    \frac{e^{i\varphi} a_0(t-s, x', \xi')}{m_1'(c_{ST})  k_1(c_{ST})}  
    \hat{\tilde{l}}_1(s, \xi')
    \diff s
    \diff \xi'
    + \text{ lower order terms },
    \end{align*}
    where we define
    \begin{align*}
    &\zeta_1 = (2 (\mu_+ - \mu_-) - (\frac{\rho_+ c_{ST}^2 }{1 - a_+(c_{ST}) b_+(c_{ST})} -\frac{\rho_- c_{ST}^2 }{1 - a_-(c_{ST}) b_-(c_{ST})} ) ),\\
    &\zeta_2 = (\frac{b_+(c_{ST}) \rho_+ c_{ST}^2 }{1 - a_+(c_{ST}) b_+v} + \frac{b_-(c_{ST}) \rho_- c_{ST}^2 }{1 - a_- (c_{ST})b_-(c_{ST})})
    \end{align*}
    as smooth functions w.r.t. $t, x', \xi'$
    with $c_{ST}$ be the Stoneley speed, $m_1(s), k_1(s)$ defined in (\ref{def_m1_s}, \ref{def_k_s}).
\end{thm}
\begin{remark}
    With $f^-$ given above and $f^+ = l + f^-$, we can construct the solution $u^-, u^+$ as discussed in Section \ref{u_construct}, which satisfy
    \begin{equation*}
    \begin{cases}
    \partial_t^2 u^{\pm} - Eu^\pm  = 0 \mod C^\infty & \mbox{in } \mathbb{R}_t \times \Omega_{\pm},\\
    [u] = l \mod C^\infty, \quad [N u] = q  \mod C^\infty&\mbox{on } \Gamma,\\
    u|_{t<0} = 0.
    \end{cases}
    \end{equation*}
    It remains to show that $u^-, u^+$ have the same microlocal behaviors as the true solution to (\ref{elastic_trans}).
\end{remark}


\begin{appendices}
\section{the principal symbol of $r$}\label{symbolr}
In this section, we explain how to compute the principal symbol of $r$, the lower order term in the hyperbolic equation in Section \ref{Rayleigh}.
It determines the transport equation satisfied by the amplitude that we constructed there.

By (\ref{def_lambda}) we have $r = e^{-1}_0 r_1$, where $r_1$ is the lower order term in the fully decoupled system. 
Recall the decoupling procedure in \cite{Taylor2017} and the diagonalization of $\Lambda$ in Section \ref{Rayleigh}. First we have
\[
W^{-1} \Lambda W = 
\begin{pmatrix} 
\widetilde{m}_1  & 0 & 0\\
0 & \widetilde{m}_2  & 0\\
0 & 0 & \widetilde{m}_3 
\end{pmatrix} 
+ R
\equiv G + R
\mod \Psi^{-\infty},
\]
where $R \in \Psi^0$ is the lower order term given by the asymptotic expansion of the product of $\Psi$DOs, see (\ref{Rij}).
We denote the entries of $R$ by $R_{ij}$, for $i,j = 1,2,3$ and will compute them in the following. 
One can see only the entry $R_{11}$ is involved in the principal symbol of $r$.
Then we decouple the system further by finding a $K^1$ in form of 
\[
K^1 =
\begin{pmatrix} 
0  & K_{12} & K_{13} \\
K_{21}  & 0  & K_{23} \\
K_{31}  & K_{32}  & 0
\end{pmatrix} 
\in \Psi^{-1}
\] 
such that
\[
(1+K^1)(G+R) (1+K^1)^{-1}= G + (K^1G - GK^1 + A) + \cdots
\]
and the off diagonal terms of $(K^1G - GK^1 + A)$ vanish. 
There exists a unique solution for $K^1$ since all eigenvalues are distinct.
After this step, the diagonal terms $R_{jj}, j = 1, 2,3$ in $R$ remains and they form the diagonal of the zero order terms of the decoupled system. 
We introduce the notations $ P^{(\alpha)} =  \partial_{(\tau,\xi)}^\alpha P$ and $ P_{(\alpha)} =  D_{(t,x)}^\alpha P$ for a matrix-valued symbol $P(t,x,\tau, \xi)$ in the following.
Then the principal symbol of $R$ are the second highest order term of the full symbol
\begin{align}\label{Rij}
\sigma(W^{*} \Lambda W) = \sigma(W^{*} (\Lambda W)) 
\sim \sum_{\alpha \geq 0} \frac{i^{|\alpha|-1}}{\alpha !}  \sigma(W^{-1})^{(\alpha)}  \sigma(\Lambda W)_{(\alpha)}\\
\sim \sum_{\alpha \geq 0} \frac{i^{|\alpha|-1}}{\alpha !}  \sigma(W^{-1})^{(\alpha)} 
(\sum_{\beta \geq 0} \sigma(\Lambda)^{(\beta)} \sigma( W)_{(\beta)})_{(\alpha)},
\end{align}
where the product of symbols are multiplication of matrices.
Observe that any terms that are multiples of $\widetilde{m}_1$ vanish along $\Sigma_R$.
It follows that
\begin{align*}
\sigma_p(R_{11}) &=  \sum_{|\alpha|=1} \sum_k  \sigma_p(W^{*})_{1k}^{(\alpha)} ((\sigma_p(\Lambda)\sigma_p(W))_{(\alpha)})_{k1} + \sigma_p(W^{*})_{1k} (\sigma_p(\Lambda )^{(\alpha)} \sigma_p(W)_{(\alpha)}))_{kj}\\
&+\sum_{k}
\sigma_p(W^{*})_{1k} (\sigma_{p-1}(\Lambda) \sigma_p(W))_{k1},
\end{align*}
where we use $\sigma_{p-1}$ to denote the symbol of second highest order. 
The third term depends on $\sigma_{p-1}(\Lambda)$, which can be computed from Appendix \ref{App_lambda}.

Additionally,
since we write $\widetilde{m}_1 \sim i e_0(\partial_t - i c_R(t,x')\sqrt{-\Delta_{x'}})$ in the principal symbol level, there is an extra term in $r_1$ besides $R_{11}$. Let $p = i(\tau - c_R(t,x')|\xi'|)$. Then we have
\[
\sigma(r_1) = \sigma_p(R_{11}) - i \sum_{|\alpha|=1} e_0^{(\alpha)} p_{(\alpha)},
\]
and therefore 
\[
\sigma_p(r) = \sigma_p(e_0^{-1}r_1) = \sigma_p(e_0)^{-1}(\sigma_p(R_{11}) - i\sum_{|\alpha|=1} e_0^{(\alpha)} p_{(\alpha)}).
\]

\section{the second highest oder term of $\sigma(\Lambda)$}\label{App_lambda}
In this section, we explain how to compute $\sigma_{p-1}(\Lambda)$, the second highest order term of the symbol $\sigma(\Lambda)$. 
We show that when the boundary is flat and the coefficients are constant (the flat case), we have $\sigma_{p-1}(\Lambda) = 0$.

First, in \cite{Stefanov2019} $w(t,x)$ is constructed from the its boundary value $w_b \equiv w|_{x^3=0}$ by the geometric optics construction with a complex phase function and an matrix-valued amplitude $a_w$. 
One can regard $w(t,x)$ as the result of applying the solution operator, an FIO with a complex phase, to $w_b$. 
In the flat case, especially we have the exact solution with the amplitude $a_w = I_3$, see (\ref{flatw}).

Then recall in \cite{Stefanov2019} we have
\begin{align*}
u = Uw, \quad h = Nu|_{x^3 =0} = M_{out} w_b = \Lambda f, \text{ with } f = u|_{x^3=0}= U_{out} w_b,
\end{align*}
where the Neumann operator $N$, $U$ are defined and computed there.
It follows that $\sigma_{p-1}(\Lambda)$ can be computed by 
\begin{align*}
\sigma_{p-1}(\Lambda)  &= \sigma_{p-1}(M_{out})\sigma_{p}(U_{out})^{-1} - \sigma_p(\Lambda)\sigma_{p-1}(U_{out}) \sigma_p(U_{out})^{-1} \\
&- (\sum_{|\alpha| = 1}  \sigma_{p}(\Lambda)^{(\alpha)}\sigma_{p}(U_{out})_{(\alpha)})\sigma_p(U_{out})^{-1},
\end{align*}
where $\sigma_{p-1}(M_{out}), \sigma_{p-1}(U_{out})$ can be computed by the asymptotic expansion of the composition of a $\Psi$DO with an FIO.
Indeed, by the Fundamental Lemma for the complex phase in \cite{Treves1980},
one can see this composition is actually a $\Psi$DO if restricted on the boundary $x_3 =0$.
More specifically, one can derive
\begin{align*}
\sigma_{p-1}(M_{out})
&= \sigma_{p-1}(NU)|_{\xi_3 = i\alpha} 
\begin{pmatrix}
I_2 & \\
& 0
\end{pmatrix}
+
\sigma_{p-1}(NU)|_{\xi_3 = i\beta} 
\begin{pmatrix}
0_2 & \\
& 1
\end{pmatrix}\\
& + 
\partial_{\xi_3} \sigma_p(NU)|_{\xi_3 = i\alpha} D_{x^3} a_w|_{x^3 = 0}
\begin{pmatrix}
I_2 & \\
& 0
\end{pmatrix}
+
\partial_{\xi_3} \sigma_p(NU)|_{\xi_3 = i\beta} D_{x^3} a_w|_{x^3 = 0}
\begin{pmatrix}
0_2 & \\
& 1
\end{pmatrix},
\end{align*}
where $a_w|_{x^3 = 0}$ can be derived from the geometric optics construction.
Similarly we have 
\begin{align*}
\sigma_{p-1}(U_{out}) = 
&= \sigma_{p-1}(U)|_{\xi_3 = i\alpha} 
\begin{pmatrix}
I_2 & \\
& 0
\end{pmatrix}
+
\sigma_{p-1}(U)|_{\xi_3 = i\beta} 
\begin{pmatrix}
0_2 & \\
& 1
\end{pmatrix}\\
& + 
\partial_{\xi_3} \sigma_p(U)|_{\xi_3 = i\alpha} D_{x^3} a_w|_{x^3 = 0}
\begin{pmatrix}
I_2 & \\
& 0
\end{pmatrix}
+
\partial_{\xi_3} \sigma_p(U)|_{\xi_3 = i\beta} D_{x^3} a_w|_{x^3 = 0}
\begin{pmatrix}
0_2 & \\
& 1
\end{pmatrix}.
\end{align*}
We note that $\sigma_{p-1}(U)$ and $\sigma_{p-1}(N)$ are involved in above computation. 
The latter can be found in \cite{Stefanov2019} and $\sigma_{p-1}(U)$ can be computed from the procedure of fully decoupling the operator $E$. 
Additionally, the second term in $\sigma_{p-1}(\Lambda)$ multiplied by $\sigma_p(W^*)$ on the left will vanish along $\Sigma_R$ in the first row, which implies it has no contribution to $\sigma_p(R_{11})$.

Notice in the flat case, we have $\sigma(N), \sigma(U)$ are homogeneous in $\tau, \xi$ of order $1$ and $a_w =I_3$. Therefore in this case $\sigma_{p-1}(\Lambda) = 0$.
\end{appendices}

\begin{footnotesize}
    \bibliographystyle{plain}
    \bibliography{microlocal_analysis}
\end{footnotesize}

\end{document}